%% file: main.tex
\documentclass[a4paper]{article}

\usepackage{amsmath,amssymb,amsthm}
\usepackage{mathrsfs}
\usepackage{authblk}

\usepackage[left=2cm,right=2cm,top=2cm,bottom=2cm,bindingoffset=0cm]{geometry}
\usepackage{nccfoots}
\usepackage[hang,flushmargin]{footmisc} 
\usepackage{hyperref}

\usepackage{pgf,tikz}

\usetikzlibrary{calc,math}
\usetikzlibrary{angles,decorations.markings}

\tikzset{mydeco/.style={pic actions/.append code=\tikzset{postaction=decorate}}}

\newtheorem{remark}{Remark}

\newtheorem{lemma}{Lemma}
\newtheorem{corollary}{Corollary}

\newtheorem{observation}{Observation}
\newtheorem{theorem}{Theorem}

\newcommand{\HH}{{\mathcal H}}
\renewcommand{\H}{\HH^1}

\DeclareMathOperator{\conv}{conv}

\newcommand{\St}{{\mathcal St}}

\title{Steiner trees with infinitely many terminals on the sides of an angle}

\author[1]{Danila Cherkashin\footnote{\href{mailto:jiocb.orlangyr@gmail.com}{jiocb.orlangyr@gmail.com}}}
\author[2]{Emanuele Paolini}
\author[3]{Yana Teplitskaya}

\affil[1]{Institute of Mathematics and Informatics, Bulgarian Academy of Sciences, Sofia, Bulgaria}
\affil[2]{Universit{\`a} di Pisa, Pisa, Italy}
\affil[3]{Laboratoire  de Math{\'e}matiques d’Orsay, Universit{\'e} Paris-Saclay, CNRS, Orsay, France}

\begin{document}

\maketitle

\begin{abstract}
The Euclidean Steiner problem is the problem of finding a set $\St$,
with the shortest length, such that $\St \cup \mathcal{A}$ is connected, where $\mathcal{A}$ is a given set in a Euclidean space. The solutions $\St$ to the Steiner problem will be called \textit{Steiner sets} while
the set $\mathcal{A}$ will be called \textit{input}. Since every Steiner set is acyclic we call it Steiner tree in the case when it is connected.
We say that a Steiner tree is \emph{indecomposable} if it does not contain any Steiner tree for a subset of the input.

We are interested in finding the Steiner set when the input consists of infinitely many points distributed on two lines. In particular we would like to find a configuration which gives an indecomposable Steiner tree.

It is natural to consider a self-similar input, namely the set $\mathcal{A}_{\alpha,\lambda}$ of points with coordinates $(\lambda^{k-1}\cos \alpha,$ $\pm \lambda^{k-1}\sin \alpha)$, where $\lambda>0$ and $\alpha>0$ are small fixed values and $k \in \mathbb{N}$.
These points are distributed on the two sides of an 
angle of size $2\alpha$ in such a way that the distances from the points to the vertex of the angle are in a geometric progression.

To our surprise, we show that in this case the solutions to the Steiner problem for $\mathcal{A}_{\alpha,\lambda}$,
when $\alpha$ and $\lambda$ are small enough, 
are always decomposable trees. More precisely, any Steiner tree for $\mathcal{A}_{\alpha,\lambda}$ is a countable union of Steiner trees, each one connecting 5 points from the input.
Each component of the decomposition can be mirrored with respect to the angle bisector 
providing $2^{\mathbb N}$ different solutions with the same length.
By considering only a finite number of components we obtain 
many solutions to the Steiner problem for finite sets composed of $4k+1$ points distributed on the two lines 
($2k+1$ on a line and $2k$ on the other line).
These solutions are very similar to the \emph{ladders} of Chung and Graham.

We are able to obtain an indecomposable Steiner tree by adding,
to the previous input, 
a single point strategically placed inside the angle.
In this case the solution is in fact a self-similar tree (in the sense that 
it contains a homothetic copy of itself). 

Finally, we show how the position of the Steiner points in the Steiner 
tree can be described by a discrete dynamical system which turns out 
to be equivalent to a 2-interval piecewise linear contraction.
\end{abstract}

\section{Introduction}

The finite Euclidean Steiner problem is the problem of finding a one-dimensional connected set $\St$ of minimal length containing a finite set of given points $\mathcal A \subset \mathbb{R}^d$.  

The history of the finite Euclidean Steiner problem is studied in the paper~\cite{brazil2014history}.
Brazil, Graham, Thomas and Zachariasen did a detailed research and discovered that the statement and basic results about the Steiner Problem were rediscovered (at least) three times: it was first stated by Gergonne in 1811, by Gauss in 1836 and by Jarn{\'\i}k and K{\"o}ssler~\cite{jarnik1934minimal}
in 1934.
The problem has become well-known as the ``Steiner problem'' after the great success of the book ``What is Mathematics?'' by Courant and Robbins~\cite{Courant1941}.

In 1980-s and 1990-s, explicit solutions to the finite Steiner problem attracted the attention of several notable mathematicians. It is worth noting that Du, Hwang and Weng~\cite{du1987steiner} completely solved the Steiner problem when $\mathcal A$ is the set of vertices of a regular polygon. Rubinstein and Thomas~\cite{rubinstein1992graham} generalizes the result when the points of $\mathcal A$ are uniformly enough distributed along a circle.

A setting which is similar to ours, can be found in the paper by Burkard and Dud{\'a}s~\cite{burkard1996steiner}, who determined the Steiner trees for the vertex of a plane angle of size at least $\pi/6$ and all points in the sides at a distance $1,\dots, n$. It turns out that apart from the points which are close to the vertex, the Steiner tree is composed of segments of length $1$ joining the points on the sides of the angle.
The results of our paper are also very similar to the results about \emph{ladders} of Chung and Graham~\cite{MR0480116} and Burkard, Dud{\'a}s and Maier~\cite{burkard1996cut}. 
A \textit{ladder} is a collection of $2n$ lattice points 
placed on the vertices of $n-1$ adjacent unit squares, forming 
a rectangle of sides $1\times(n-1)$.
It turns out that when $n$ is odd the Steiner tree is full, which means that terminal points are reached by a single edge of the tree.
While if $n > 2$ is even the Steiner tree is not full. 
By splitting the tree in the terminals of order two 
one can \emph{decompose} the Steiner tree into 
the union of small Steiner trees, each one containing 
only two or four terminal points.
Notice that this is a case (relevant to our quest) where a full Steiner tree with an arbitrarily large number of points is presented. 
However, as the number of points is going to infinity also 
the length of the Steiner tree is going to infinity, hence this example is not suitable to construct a tree with infinitely many leaves.

Finding the Steiner tree for a finite set of points is a well-known problem in computational geometry.
Garey, Graham and Johnson~\cite{garey1977complexity} proved that the Steiner problem is NP-hard. Rubinstein, Thomas and Wormald~\cite{rubinstein1997steiner} proved that such 
a complexity persists even when the terminals 
are constrained to lay along two parallel lines and also in the case when the terminals are placed on the sides of an angle with size which is smaller than $2\pi/3$ (see Theorem 3 in~\cite{rubinstein1997steiner}).

More recently, a generalized setting for the Steiner problem was given by Paolini and Stepanov~\cite{paolini2013existence}: the ambient space $X$ can be any connected complete metric space with the Heine--Borel property (closed bounded sets are compact) and the given set of points can be any compact subset of the ambient space.
In this setting there always exists a set $\St$, with minimal $1$-dimensional Hausdorff measure $\H$, such that $\St\cup \mathcal A$ is connected.
Such $\St$ is proven to be acyclic. If $\St$ itself is connected we call it \textit{Steiner tree for $\mathcal{A}$}. 
The set $\mathcal{A}$ is called an \textit{input} for the Steiner problem.
See Section~\ref{sec:steiner} for precise definition and properties of Steiner trees.

Following the paper \cite{paolini2013existence}, the question 
of finding non-trivial examples of Steiner trees for infinitely many 
points was raised.
Of course, it is easy to find solutions with infinitely many points by splitting a finite tree into infinitely many pieces. 
For example given any compact set of points on the real line, $A\subset\mathbb R$, 
the solution to the Steiner problem with input $A$ is clearly given by 
the set $\St = [\min A,\max A]\setminus A$.
If $A$ is compact the solution $\St$ is composed of a countable number of disjoint open segments each with end points on $\partial A$.
Much more difficult is to find an infinite Steiner tree which is \emph{indecomposable} in the sense 
that it is not possible to split the tree $\St$ into two connected pieces 
$\St_1$ and $\St_2$ so that both $\St_1$ and $\St_2$
are Steiner trees over a subset of the set $\mathcal A$ of terminal points.

\begin{figure}[ht]
    \centering
    \input{anglepict/1trees.tex}
    \caption{The left part contains two Steiner trees connecting the vertices of a square; the right part illustrates an example of a self-similar solution $\Sigma(\lambda)$.}
    \label{pict:trees}
\end{figure}
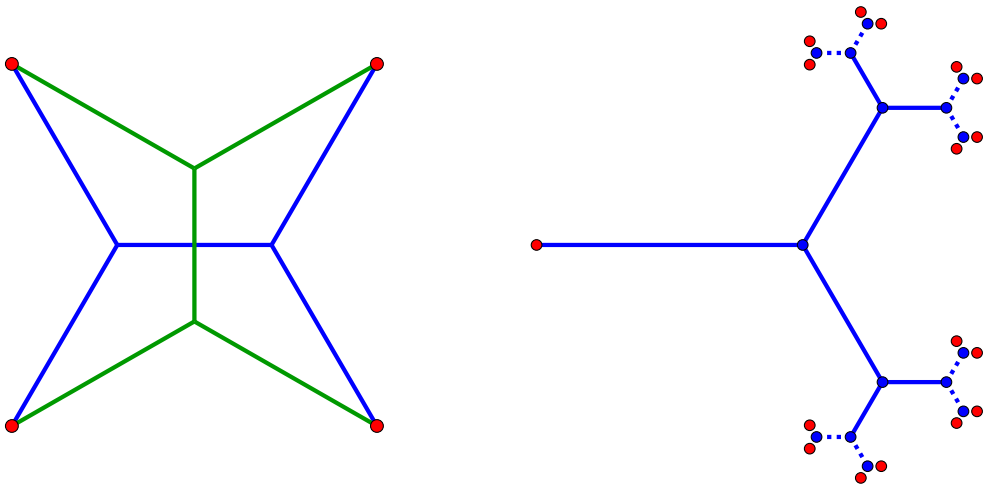

The first example of an infinite, indecomposable Steiner tree was given in~\cite{paolini2015example}.
However it was not self-similar and its input set $\mathcal{A}$ had zero Hausdorff dimension.
The next natural question was to understand if the same result was true 
for the set of leaves of a self-similar binary tree $\Sigma(\lambda)$
composed of a segment of length $1$ which splits into two segments 
of length $\lambda$ each of which splits into two segments 
of length $\lambda^2$ and so on\dots
The splitting points are \emph{regular tripods} meaning that there are three segments joining in the point forming equal angles of size $2\pi/3$
(see Figure~\ref{pict:trees}, right-hand side picture). 
The input set $\mathcal{A}(\lambda)$ is composed of a root point (the first end-point of the unit segment) and all the accumulation points of the smaller and smaller segments.
Note that, if $\lambda < \frac 1 2$ the set $\mathcal{A}(\lambda)$ turns out to be a self-similar, fractal set, with Hausdorff dimension $-\frac{\ln 2}{\ln \lambda}>0$.
Recently, in~\cite{cherkashin2023self,paolini2023steiner}, $\Sigma(\lambda)$ was proven to be the unique Steiner tree for the set $\mathcal{A}(\lambda)$ if $\lambda$ is small enough.

Clearly every connected subset $S$ of $\Sigma(\lambda)$ is the unique Steiner tree for the set of its endpoints and corner points (otherwise one can replace $S$ with a better competitor in $\Sigma(\lambda)$).

So $\Sigma(\lambda)$ can be regarded as a \textit{universal graph} for Steiner trees in the sense that it contains a Steiner subtree with any given possible finite combinatorics. 
This fact was recently used in~\cite{basok2018uniqueness} to show that the set of $n$-point configurations in $\mathbb{R}^d$ with a unique Steiner tree is path-connected (as a subset of $\mathbb{R}^{dn}$).

In this paper we are interested in finding an explicit example of a Steiner tree connecting a countable number of points distributed along two lines forming an angle in the plane. 
The points will be equally distributed on the two sides of the angle in such a way that the distances of the points from the vertex of the angle are in a geometric progression.

A configuration like this was suggested in the Master's Thesis of Letizia Ulivi, where an existence theorem for Steiner trees was proven in the case when the input set was considered to be compact and countable, see \cite{PaoUli09}. 

Examples of this kind are interesting because the set of terminals is distributed on the boundary of a convex set. 
In this setting the Steiner tree is splitting the convex set into many \emph{regions} and hence determines a \emph{partition}.
The Steiner Problem becomes a problem of \emph{minimal partitions} and 
can hence be studied in the framework of geometric measure theory 
by using calibrations (see \cite{Morgan2005,MarcheseMassaccesi2016,CarioniPluda2021}).

To our surprise we have found that, under suitable assumptions (Theorem~\ref{theo:main2}), the Steiner tree for such a set of points is decomposable into rescaled copies of a Steiner tree connecting only 5 points. 
However by adding a pair of terminal points in a precise position, outside the geometric progression, we are able 
to force the Steiner tree to be indecomposable, see Remark~\ref{rem:main}. 

The main difficulty in the proof is the fact that the set of locally minimal trees may have a very complicated structure (see Remark~\ref{rem:localminimas}) and our methods heavily use the global optimality of a solution.

\begin{figure}[h]
    \centering
    \input{anglepict/2defA.tex}
    \caption{The input sets $\mathcal A_1=\{A_\infty, A_1,B_1,A_2,B_2,\dots\}$, $\mathcal A_0=\mathcal A_1\cup [A_0 B_0]$. 
    Here $\alpha = \frac \pi {36}$ and $\lambda = \frac 1 2$.}
    \label{pict:setup}
\end{figure}
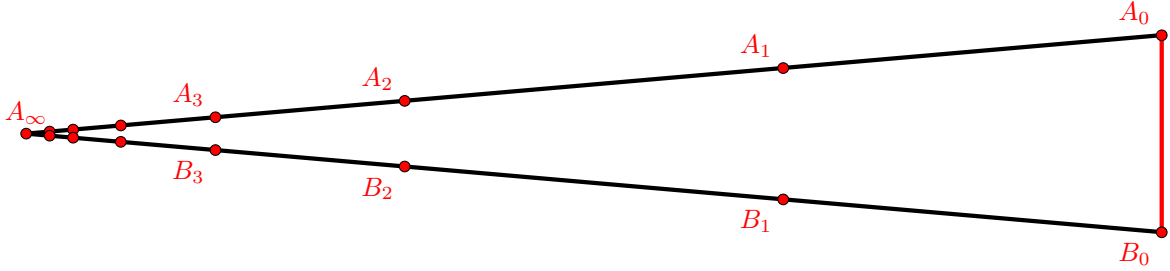

\paragraph{Our setup.}      
Consider a positive $\lambda \leq \frac 1 2$ and an angle of size $2\alpha < \frac \pi 3$ in the plane. 
Let $A_\infty$ be the vertex of the angle, and $A_1, B_1$ be points on distinct sides of the angle at a unit distance from $A_\infty$ (so that the middle point of the segment $[A_1B_1]$ has distance $\cos \alpha$ from $A_\infty$).
For $k \geq 1$, define $A_{k+1} \in [A_{k}A_\infty]$ and $B_{k+1} \in [B_{k}A_\infty]$ recursively by $|A_{k+1}A_\infty| = \lambda \cdot |A_{k}A_\infty|$ and $|B_{k+1}A_\infty| = \lambda \cdot |B_{k}A_\infty|$ (see Figure~\ref{pict:setup}). Define $A_0$, $B_0$ on the rays $[A_\infty A_1)$ and $[A_\infty B_1)$, respectively, in such a way that 
\[
|A_0A_\infty| = |B_0A_\infty| = \frac{1}{\lambda} - \frac{\tan(\alpha)}{\sqrt{3}\lambda}.
\]
Let 
\[
\mathcal{A}_1 = \mathcal{A}_1(\alpha,\lambda) := \{A_\infty \} \cup \bigcup_{n=1}^{\infty} \{A_n, B_n\} 
\]
and
\[
\mathcal{A}_0 = \mathcal{A}_0(\alpha,\lambda) := \mathcal{A}_1(\alpha,\lambda) \cup [A_0B_0].
\]

The following theorem is the main result of the paper. 

\begin{theorem} \label{theo:main2}
Suppose $0< \alpha < \frac{\pi}{6}$ and $0 < \lambda \leq \frac 1 2$ satisfy that 
\begin{equation}\label{eq:438466}
  \sqrt \lambda  < \frac{\cos \left(\frac{\pi}{3} + \alpha \right)}{\cos \left(\frac{\pi}{3} - \alpha \right)}.
\end{equation}
Then every solution $\St_1$ to the Steiner problem for $\mathcal{A}_1$ is the union of full trees on 5 terminals (see Fig.~\ref{pict:mainth2}). The length of a solution has the following explicit formula
\[
 \H(\St_1) =  \left | \cos \alpha + \sqrt{3} \sin \alpha + \frac{2\lambda}{1-\lambda^2} e^{\frac{\pi i}{6}} \sin \alpha + 
 \frac{2 \lambda^2}{1-\lambda^2} e^{-\frac{\pi i}{6}}\sin \alpha \right |.
\]
\end{theorem}

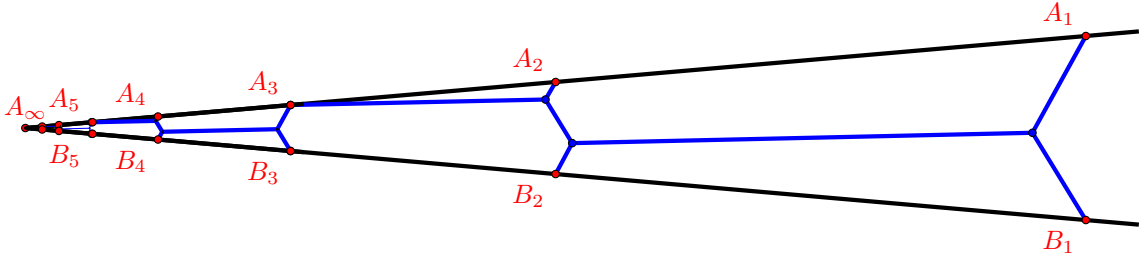
\begin{figure}[h]
    \centering
    \input{anglepict/3.3mainT.tex}
    \caption{One of the solutions to the Steiner Problem 
    for $\mathcal A_1$. See Theorem~\ref{theo:main2}. Here $\alpha = \frac \pi {36}$ and $\lambda = \frac 1 2$.}
    \label{pict:mainth2}
\end{figure}

Note that every Steiner tree for $\mathcal{A}_1$ is not full and so can be decomposed into solutions for 5-terminals subset of $\mathcal{A}_1$. 
Each of the solutions on 5 points can be mirrored with respect to the angle bisector, so we have $2^{\mathbb N}$ different solutions, all with the same length.
Since every subtree of a Steiner tree is itself a Steiner tree with respect to its terminal points, we can state the following corollary, which holds in 
the framework of finite Steiner trees.

\begin{corollary} \label{cor:finite}
Suppose $\alpha$ and $\lambda$ are given with the assumptions of Theorem~\ref{theo:main2}.
Then the Steiner problem for the finite set $\{A_1,\dots,A_{2k+1},B_1,\dots,B_{2k}\}$ has $2^k-1$ different solutions each of which is the union of rescaled and mirrored copies of the Steiner tree on 5 terminals: $\{A_1, A_2, A_3, B_1, B_2\}$.
The length also has an explicit formula.
\end{corollary}

We would like to slightly modify the input set $\mathcal{A}_1$ to obtain an indecomposable solution.
It turns out that a simple way is to consider the input $\mathcal{A}_0$, which is obtained from $\mathcal{A}_1$ with the additional segment $[A_0B_0]$ (see definition above).
We can then replace the segment with a single point $x$ placed in a precise position on the segment $[A_0B_0]$ or with two points $x_1,x_2$ placed on the two sides of the angle in such a way that the Steiner tree has a triple 
point at the point $x$.

\begin{theorem} \label{theo:main}
Suppose that $\alpha < \frac{\pi}{6}$ and $\lambda \leq 1/2$ satisfy~\eqref{eq:438466}.
Then the following statements hold true.
\begin{itemize}
    \item [(i)] The Steiner problem for $\mathcal{A}_0$ has exactly 2 solutions $\St_0^1$ and $\St_0^2$ which are trees. 
    The tree $\St_0^2$ is the reflection of $\St_0^1$ with respect to the angle bisector. Both solutions are indecomposable. The point $\{x\}= \St_0^1 \cap [A_0B_0]$
    (and the corresponding symmetric point $\St_0^2 \cap [A_0B_0]$)  
    has a distance 
    $\frac{1}{\lambda + 1} \sin \alpha$ from the midpoint of $[A_0B_0]$.
    \item [(ii)] The length of the solutions is
    \[
    \H(\St_0^1) = \H(\St_0^2) = \frac{\cos\alpha}{\lambda} - \frac{\sin\alpha}{\sqrt{3}\lambda} + \frac{\sqrt{3}}{1 - \lambda} \sin \alpha.
    \]
    \item [(iii)] Each $\St_0^j$, $j=1,2$, is a self-similar tree, in the 
    sense that $f_2(\St_0^j) \subset \St_0^j$, where $f_2$ is the homothety with center $A_\infty$ and ratio $\lambda^2$.
\end{itemize}
\end{theorem}

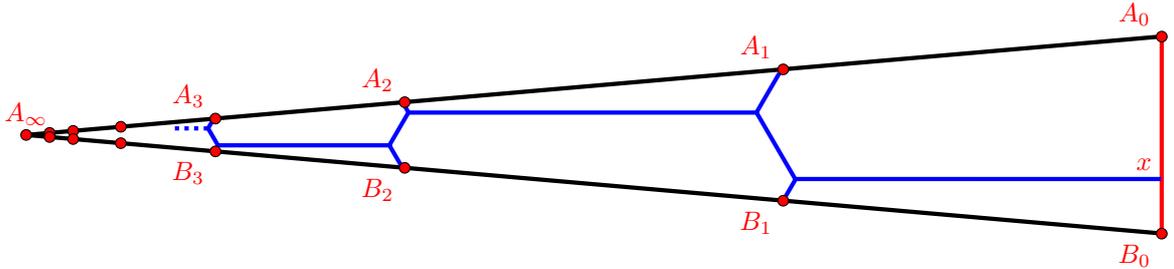
\begin{figure}[h]
    \centering
    \input{anglepict/3theorem.tex}
    \caption{One of the two Steiner trees $\St_0^1$, $\St_0^2$ 
    for the set $\mathcal A_0$, see Theorem~\ref{theo:main}.
    The other solutions is the reflection with respect to the 
    angle bisector.
    Here $\alpha = \frac \pi {36}$ and $\lambda = \frac 1 2$.}
    \label{pict:mainth}
\end{figure}

\begin{remark}\label{rem:main}
If we consider a Steiner tree $\St_0$ for $\mathcal A_0$ as 
stated in the previous theorem and insert a Steiner point 
in the vertex $x$ with two additional edges $[xA_0']$ and 
$[xB_0']$ so that $A_0'$ and $B_0'$ are on the edges of the angle 
$[A_\infty,A_1)$ and $[A_\infty,B_1)$, respectively, we obtain
a full Steiner tree for the set of terminals $\mathcal A_0' = \mathcal A_1 \cup \{A_0',B_0'\}$ placed on the sides of the 
triangle $\triangle A_0'A_\infty B_0'$.
\end{remark}

\paragraph{Structure of the paper.} Section~\ref{sec:NB} contains the notation and some general results we will use in the proofs. 
These results are well-known for a finite input and we translate them to an infinite setup. For the sake of completeness, the proofs are collected in the Appendix.
In Section~\ref{sec:lemmas} we describe the rough structure of a Steiner tree for both input sets $\mathcal{A}_0$ and $\mathcal{A}_1$.
The proofs of Theorem~\ref{theo:main2} and~\ref{theo:main} are given in Sections~\ref{sec:proof1} and~\ref{sec:proof2}, respectively.
In Section~\ref{sec:dynamics} we present a formulation to 
find a ``locally minimal'' competitor by means of a dynamical system.
Section~\ref{sec:open} collects open questions.
In the Appendix we have collected the proofs of well known results.

\section{Notation and basics} \label{sec:NB}

\subsection{Notation}

If $T$ is a subset of $\mathbb R^d$
we denote by $\overline{T}$, $\partial T$ and $\conv T$, 
respectively the closure, the boundary and the convex hull of $T$.
For $\rho>0$ and $C\in \mathbb R^d$ we denote by 
$B_\rho(C)$ the open ball centered in $C$ with radius $\rho$.
If $T$ is a set we denote by $B_\rho(T) = \bigcup_{C\in T} B_\rho(C)$
the open $\rho$-neighbourhood of $T$.
We denote with $\H(T)$ the one-dimensional Hausdorff measure 
of the set $T$.

Given $B,C,D \in \mathbb{R}^d$ we denote by $\lvert BC\rvert=\lvert C-B\rvert$ the Euclidean distance between $B$ and $C$. 
Also, $[BC]$ denotes the closed segment with endpoints at $B$ and $C$,
while $]BC[$ denotes the open segment (without the endpoints).
We denote with $(BC)$ the line passing through $B$ and $C$
and with $[BC)$ the half-line (ray) starting at $B$ and passing through $C$.
We denote with $\angle BCD$ the convex angle with vertex $C$ and sides 
containing the points $B$ and $D$ and with $\triangle BCD$ the triangle 
with vertices $B$, $C$ and $D$.

We let $\mathbb N= \{1,2,3,\dots\}$ be the set of positive whole numbers.
\subsection{Graphs and immersions}

A \emph{graph} $\Gamma$ is a pair $(V,E)$ where $V$ is any set and 
$E$ is a set such that every $e\in E$ is of the form $e=\{v,w\}$ with distinct $v,w\in V$.
The elements of $V$ are called \emph{vertices} of $\Gamma$ and the elements of $E$
are called $\emph{edges}$. If $v$ is a vertex of $\Gamma$ we define the order (or degree)
of $v$ as $\#\{e\in E\colon v\in e\}$ i.e.{} the number of edges joining in $v$.

If each vertex of $\Gamma$ has finite order, we say that $\Gamma$ is a 
\emph{locally-finite} graph.
If $E$ and $V$ are both finite, we say that $\Gamma$ is a \emph{finite} graph.
A \emph{finite path} $\gamma$ is an alternating sequence of vertices and edges: 
$v_0, e_1, v_1, e_2, v_2, \dots, e_n, v_n$ 
such that $e_k = \{v_{k-1},v_{k}\}$.
If $v_i\neq v_j$ for all $(i,j)$ with $i < j$ and $(i,j) \neq (0,n)$ we say that the path is \emph{simple}.
A (simple) path with $v_0 = v_n$ will be called a (simple) \emph{cycle}.
A path with no edges will be called \emph{trivial}
otherwise it will be called \emph{non-trivial}.
If for all pair of vertices $v,w \in V$ there is a finite path with first 
vertex $v$ and last vertex $w$ we say that the graph is \emph{connected}.
If there is no non-trivial cycle we say that the graph is a 
\emph{forest}. A connected forest is called a \emph{tree}.

Given $\varphi \colon V \to \mathbb R^d$ we say that the set
\[
  S_\varphi = \bigcup_{\{v,w\}\in E} [\varphi(v)\, \varphi(w)] \subset \mathbb R^d
\]
is a \emph{geodesic immersion} (or simply an \emph{immersion}) of $\Gamma$ in $\mathbb R^d$. The points $\varphi(v)$ are called \emph{immersed vertices} while the segments $[\varphi(v) \varphi(w)]$ with $\{v,w\}\in E$ 
are called \emph{immersed edges}.
The \emph{length} of an immersed graph is defined by
\[
  \ell(\varphi) = \sum_{\{v,w\}\in E} \lvert \varphi(v)\, \varphi(w)\rvert.
\]
Clearly one has $\ell(\varphi) \ge \H(S_\varphi)$.

If the graph is finite, the map $\varphi$ is injective 
and the segments of the immersion have pairwise disjoint interior 
(i.e.\ the intersection of $]\varphi(v)\,\varphi(w)[$ and $]\varphi(v')\,\varphi(w')[$ is empty whenever $\{v,w\}\neq\{v',w'\}$)
we say that the immersion is an embedding.

In the following we will often refer to immersed graphs, immersed vertices and immersed edges 
simply as graphs, vertices and edges.

We will say that an immersed graph $S_\varphi$ is \emph{full*} when the abstract graph is connected 
and all the convex angles between any two edges adjacent to the same vertex, are equal to $2\pi/3$. 
In particular in a \emph{full*} tree all vertices have order at most three.
The \emph{wind-rose} of an immersed graph is the union of all lines 
passing through the origin and parallel to any edge of 
the graph. 
Clearly, in the planar case $d = 2$, the wind-rose of a \emph{full*} immersed graph comprises at most 
three different lines (so the corresponding directed segments have at most six directions).

We will say that an immersed graph $S_\varphi$ is \emph{full} if it is \emph{full*} and there are no vertices of order two.

\subsection{General Steiner trees}
\label{sec:steiner}

If $S$ is a subset of a topological space, we say that $S$ is \textit{connected} if it is not the disjoint union of two non-empty closed sets. We say that $S$ is \textit{path-connected} if
given any two points of $S$ there is a continuous curve joining them in $S$.
We say that $S$ contains no loop if no subset of $S$ is homeomorphic to the circle $\mathbb S^1$. 
We say that $S$ is a topological tree if $S$ is connected and contains no loop.

Let $\mathcal{A}\subset \mathbb R^d$ and $\St\subset \mathbb R^d$. 
We say that $\St$ is a \emph{Steiner set} for $\mathcal{A}$ if $\St\cup \mathcal{A}$ is connected and 
$\H(\St) \le \H(S)$ whenever $S\subset \mathbb R^d$ and $S\cup \mathcal{A}$ is connected.
We say that $\St$ is a \emph{Steiner tree} if $\St$ is a Steiner set and is itself connected.

The following theorem collects some of the general results from \cite{paolini2013existence}.
\begin{theorem} \label{th:PaoSte}
If $\mathcal{A}\subset \mathbb R^d$ is a compact set then there exists a Steiner set 
$\St$ for $\mathcal{A}$. 
Moreover if $\St$ is a Steiner set for $\mathcal{A}$ and $\H(\St)<+\infty$ then 
\begin{enumerate}
    \item $\St\cup \mathcal A$ is compact;
    \item $\St\setminus \mathcal A$ has at most countable many connected components, 
    and each of them has positive length;
    \item $\overline \St$ contains no loop (homeomorphic image of $\mathbb S^1$);
    \item the closure of every connected component of $\St$ is a topological tree 
    with endpoints on $\mathcal A$ (in particular it has at most countable number of branching points)
    and has at most one endpoint on each connected component of $\mathcal A$;
    \item if $\mathcal A$ is finite then $\overline \St = \St\cup \mathcal A$ is an embedding of a finite tree;
    \item for almost every $\varepsilon>0$ the set $\St\setminus B_\varepsilon(\mathcal A)$
    is an embedding of a finite graph.
\end{enumerate}

\end{theorem}

We need to strengthen the result of the previous theorem to the case when $\mathcal{A}$ is countable.
In that case it is useful to prove that the Steiner set $\St$ is locally finite also in the neighborhood 
of the isolated points of $\mathcal{A}$.
The proof of the following statements can be found in the Appendix.

\begin{lemma} \label{lm:isolated}
    Let $x_0\in \mathcal{A}$ be an isolated point of the compact set $\mathcal A$.
    Let $\St$ be a Steiner set for $\mathcal A$ with $\H(\St)<+\infty$.
    Then there exists $\rho>0$ such that $\overline{\St \cap B_\rho(x_0)}$ is an embedding of a finite graph.
\end{lemma}

\begin{corollary} \label{cor:isolated}
Let $\mathcal A\subset \mathbb R^d$ be a compact set.
Let $\mathcal A'$ be the set of accumulation points of $\mathcal A$ 
and $\mathcal A_0 = \mathcal A\setminus \mathcal A'$ be the set 
of isolated points of $\mathcal A$.
Let $\St$ be a Steiner set for $\mathcal A$ with $\H(\St)<+\infty$.
Then, for almost every $\rho>0$ the set $\St\setminus B_\rho(\mathcal A')$
is an embedding of a finite graph.
\end{corollary}

The previous corollary allows us to use the terminology of finite Steiner trees for 
a general Steiner tree $\St$ with input $\mathcal A$. 
We will say that a point $x\in \St\setminus \mathcal A$ is a \emph{Steiner point} 
if there is $\rho>0$ such that $\St \setminus B_\rho(\mathcal A')$ is a finite 
embedded graph and $x$ is a vertex of this graph.
We will say that a segment $[x,y]\subset \St\setminus\mathcal A'$ 
is an \emph{edge} of the Steiner tree $\St$ 
if there exists $\rho>0$ such that $[x,y]$ is an embedded edge of the finite embedded
graph $\St \setminus B_\rho(\mathcal A')$.

Notice also that if $\St$ is a Steiner set for $\mathcal A$ then $\St\subset \conv \mathcal A$.
In fact we know that the projection $\pi \colon \mathbb R^d \to \conv \mathcal A$ is a $1$-Lipschitz map,
which implies that $\H(X) \le \H(\pi(X))$ for every measurable set $X\subset \mathbb R^d$. 
Thus for $\St' = \pi(\St)$ we have that $\St'\cup \mathcal A=\pi(\St\cup \mathcal A)$
and hence $\St' \cup \mathcal A$ is connected and $\H(\St')\le \H(\St)$. 
This implies that $\St'$ is itself a Steiner set for $\mathcal A$. 
But then $\St'=\St$, otherwise one would have $\H(\St')<\H(\St)$ which is a contradiction.

\begin{remark} \label{rem:exitsandgood}
In our setting $\mathcal A'$ consists of single point $\{A_\infty\}$. 
By Theorem~\ref{th:PaoSte} a solution $\St_i$ to the Steiner problem with input $\mathcal{A}_i$, $i=0,1$ exists; clearly it has a finite length.
By Corollary~\ref{cor:isolated} outside of a neighbourhood of $\{A_\infty\}$ the set $\St_i$ inherits the properties of a finite Steiner tree.
\end{remark}

\subsection{Maxwell's length formula}

In the following lemma we identify the Euclidean plane as the complex plane 
$\mathbb{C}$.
The proof of the following Lemma can be found in the Appendix.

\begin{lemma}[Maxwell-type formula] \label{lm:Maxwell}
Let $S_\varphi$ be a full*, finite, immersed tree and let 
$p_1,p_2,\dots, p_n \in \mathbb C$ be the vertices 
of order less than three. Suppose $n>1$.
If $p_k$ is vertex of order one we define 
$c_k$ as the unit complex number representing the outer direction of 
the unique edge of $p_k$.  
If $p_k$ is a point of degree two then the two edges define an angle of $\frac 2 3 \pi$ 
and we let $c_k$ be the unit complex number with the direction, away from $p_k$, of the third implied edge which would complete 
a regular tripod in $p_k$.

Then one can write the length of the tree as
\begin{equation}\label{eq:maxwell}
\ell(\varphi) = \sum_{k=1}^n \bar{c_k} p_k.
\end{equation}
In particular the right hand side of~\eqref{eq:maxwell} is a real number.
\end{lemma}

\begin{remark}
    If $\sum p_k$ converges absolutely then we can pass to the limit in~\eqref{eq:maxwell}.
    Clearly when $\{p_1, p_2, \dots \} = \mathcal{A}_1$ this sum converges absolutely if we choose $A_\infty$ to be the origin of the complex plane.
\end{remark}

\subsection{Melzak reduction}

The following lemma is an immediate corollary of Ptolemy's theorem
(see the Appendix for the proof).

\begin{lemma}[Melzak's reduction,~\cite{melzak1961problem}] \label{lm:Melzak}
Let $\St$ be a Steiner tree for $\mathcal{A} \subset \mathbb{R}^2$.
If $\St$ contains a branching point $q$ which is adjacent to two isolated points of $\mathcal{A}$, say $p_1$ and $p_2$, 
and the circle passing through $p_1,p_2$ and $q$ contains no other point of $\St$ apart 
from the two edges $[p_1q]$ and $[p_2q]$,
then the length $\H(\St)$ is equal to the length of the tree $\St'$ obtained from $\St$ 
by removing the points $p_1$ and $p_2$ and the edges $[p_1q]$ and $[p_2q]$ and prolonging the remaining edge arriving 
in $q$ up to a new vertex $p$ such that $p p_1 p_2$ is an equilateral triangle and (among the two possibilities) $p$ and $q$ are on opposite sides of the line  $(p_1,p_2)$.
\end{lemma}

In the following we will call $p$ the \textit{Melzak point} for $p_1$ and $p_2$.

\subsection{Convexity of length}

Let $G = (V,E)$ be an abstract graph and $\varphi$ an immersion of $G$ into $\mathbb R^d$
with $\ell(\varphi)<+\infty$.
If $V_0\subset V$ is fixed, we say that the immersion $\varphi$
is \emph{locally minimal} with $V_0$ fixed if 
there exists $\delta>0$ such that 
$\ell(\varphi) \le \ell(\psi)$
for every map $\psi\colon V\to \mathbb R^d$ with 
$\psi(v)=\varphi(v)$ when $v\in V_0$ and
$\lvert \psi(v)\, \psi(v)\rvert < \delta$
for all $v\in V$.

The following theorem shows that the convexity of the length function implies that the directions of the edges of a locally minimal immersion are uniquely determined.
Often this is enough to deduce the uniqueness of locally, and hence globally minimal immersion of a given graph.

\begin{theorem} \label{th:uniq}
Let $\Gamma = (V,E)$ be a graph and $\varphi, \psi \colon V \to \mathbb R^d$ be two locally minimal immersions of $G$ into $\mathbb{R}^d$
with fixed $V_0 \subset V$. Suppose that $\ell(\varphi) < +\infty$.
Then $\ell(\psi) = \ell(\varphi)$ and one has that the edges of $\varphi$ are parallel to the corresponding edges of $\psi$.
More precisely, given any edge $\{v,w\}\in E$ either it has zero length in one of the immersions ($\varphi(v)=\varphi(w)$
or $\psi(v)=\psi(w)$) or there exists $t > 0$ such that $\varphi(v)-\varphi(w) = t(\psi(v)-\psi(w))$.
\end{theorem}

\section{Preliminary lemmas} \label{sec:lemmas}

\begin{lemma}\label{lm:473988}
Let $\mathcal{A}$ be a compact subset of the Euclidean plane and let $\mathcal W$ be a closed geometrical angle 
of size at least $\frac 2 3 \pi$ such that $\mathcal{A} \cap \mathcal W = \emptyset$. 
Let $\St$ be a Steiner tree for $\mathcal{A}$ with $\H(\St)<\infty$. Then $\St \cap \mathcal W$ is either a segment or an empty set.
\end{lemma}

\begin{proof}
First of all we show that there are no Steiner points in $\mathcal W$. 
This is because in a Steiner point at least one of the three directions is contained in the angle $V$ and hence in that direction another Steiner point should be found (because there are no points of $\mathcal A$ in $V$), 
which is also inside the angle and further away from the vertex. 
These Steiner points should have an accumulation point on $A$ which is not possible since $A \cap \mathcal W = \emptyset$.

Since there are no Steiner points in $\mathcal W$ we can state that $\St \cap \mathcal W$ is composed of line segments with 
end points on $\partial \mathcal W$. Assume that there are two such segments $[A_1B_1]$ and $[A_2B_2]$, where $A_1,A_2$ belong to the same side of $\mathcal W$ while $B_1,B_2$ are on the other side. 
Let $W$ be the vertex of the angle.
Without loss of generality we assume that $|A_1B_1| \leq |A_2B_2|$. Then both $|A_1A_2|$ and $|B_1B_2|$ are shorter than $|A_2B_2|$, because the longest side of the triangle $A_2B_2W$ is $[A_2B_2]$ and $[A_1A_2],[B_1B_2]$ are subsets of $[A_2W],[B_2W]$, respectively. 

Since $\St$ is connected there is a path in $\St\setminus \mathcal W$ which joins the two segments $[A_1 B_1]$ and $[A_2 B_2]$. 
Hence if we remove the segment $[A_2 B_2]$ and replace it with 
$[A_1B_1]$ or $[A_2B_2]$ we are able to obtain a connected set 
which is shorter than the original tree $\St$ but with the same 
terminals. 
This would contradict the minimality of $\St$  and conclude the proof.
\end{proof}

\begin{lemma}\label{lm:321066}
Suppose that $\alpha < \frac{\pi}{6}$ and $\lambda$ satisfy \eqref{eq:438466}.
Let $\St$ be a Steiner tree for the set $\mathcal A_0$ or $\mathcal{A}_1$.
Then for all $k\in \mathbb N$ there exist two disjoint convex sets $\mathcal C_k$ and $\mathcal D_k$ 
such that $A_j,B_j\in \mathcal C_k$ for all $j\le k$, $A_j,B_j\in \mathcal D_k$ for all $j>k$,
and $\St\setminus (\mathcal C_k \cup \mathcal D_k)$ is a line segment with one end point on $\partial \mathcal C_k$ and one 
end point on $\partial \mathcal D_k$.
\end{lemma}
\begin{proof}
Let us show that the assumptions of the lemma imply, for all $k\in\mathbb N$, the existence of points $W_1 \in [B_kB_{k+1}]$ and $W_2 \in [A_kA_{k+1}]$ such that the two angles $\angle A_kW_1A_{k+1}$ and $\angle B_kW_2B_{k+1}$ have size at least $2\pi/3$ (see Fig.~\ref{fig:V1}).
Let $W_1$ be the point in $[B_k,B_{k+1}]$ such that $\triangle A_k A_\infty W_1$ is similar to $\triangle W_1 A_\infty A_{k+1}$, 
which is equivalent to say that $|A_\infty A_k| \cdot |A_\infty A_{k+1}| = |A_\infty W_1|^2$. 
So $|A_\infty W_1| = \lambda^{k-1/2}$. Let $2\gamma := \angle A_kW_1A_{k+1}$.
By the sine rule for $\triangle A_\infty A_k W_1$
\[
\frac{\sqrt{\lambda}}{\sin \angle A_\infty A_k W_1} = \frac{1}{\sin \angle A_\infty W_1 A_k}.
\]
Notice that since $\angle A_\infty A_k W_1=\angle A_\infty W_1 A_{k+1}$ we have $\pi = 2\alpha + 2\gamma + 2\angle A_\infty A_k W_1$ hence 
$\angle A_\infty A_k W_1 = \frac{\pi} 2 - \alpha - \gamma$ and 
$\angle A_\infty W_1 A_k = 2\gamma + \angle A_\infty W_1 A_{k+1} = 
2\gamma + \frac \pi 2 - \alpha - \gamma = \frac \pi 2 -\alpha + \gamma$.
Thus
\[
\sqrt{\lambda} = \frac{\cos (\alpha + \gamma)}{\cos (\alpha - \gamma)}.
\]
The right-hand side is decreasing in $\gamma$ since
\[
\frac{d}{d \gamma} \frac{\cos (\alpha + \gamma)}{ \cos (\alpha - \gamma)} 
= - \frac{\sin 2\alpha}{\cos^2(\alpha - \gamma)}.
\]
Hence~\eqref{eq:438466} implies $\gamma > \frac{\pi}{3}$ and $\angle A_kW_1A_{k+1} > \frac{2\pi}{3}$. Define $W_2$ as the reflection of $W_1$ with respect to the angle bisector.

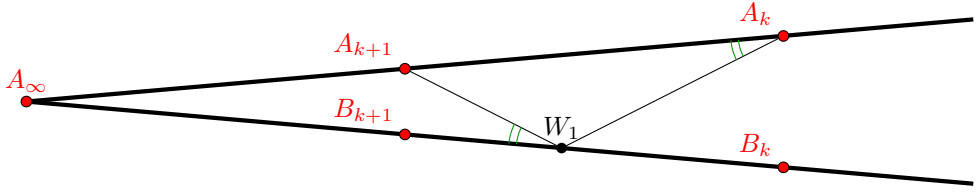
\begin{figure}[h]
    \centering
    \input{anglepict/3.5klemma.tex}
    \caption{Definition of the point $W_1$ in Lemma~\ref{lm:321066}}
    \label{fig:V1}
\end{figure}

Now take $\mathcal{W}_1$ and $\mathcal{W}_2$ to be any angles of size $2\pi/3$ with vertices $W_1$ and $W_2$ contained in $\angle A_kW_1A_{k+1}$ and $\angle B_kW_2B_{k+1}$, respectively, so that $\mathcal{W}_1, \mathcal{W}_2$ have empty intersection with the terminal set.
Lemma~\ref{lm:473988} assures that each of the two angles $\mathcal{W}_1$ and $\mathcal{W}_2$ contains a single line segment of $\St$: denote this segments respectively by $S_1$ and $S_2$. Since $\St \subset \conv \mathcal{A}$ and the segment $[W_1W_2]$ splits $\conv \mathcal{A}$ into two parts, both segments $S_1$ and $S_2$ must intersect $[W_1W_2]$ and hence $S_1\cup S_2$ must be itself a single line segment.
So the statement is proven by taking $\mathcal C_k$ and $\mathcal D_k$ to be the connected components of $\mathbb{R}^2 \setminus (\mathcal{W}_1 \cup \mathcal{W}_2)$.
\end{proof}

\begin{lemma} \label{lm:parallelogrammes}
     Suppose that $\alpha < \frac{\pi}{6}$ and $\lambda\le \frac 1 2$ satisfy \eqref{eq:438466}.
     \begin{itemize}
         \item[(i)] Any full* Steiner tree $\St_1$ for $\mathcal A_1$ is uniquely determined by its wind-rose.
         \item[(ii)] Any Steiner set for $\mathcal A_0$ is a tree.
         \item[(iii)] Any full* Steiner tree $\St_0$ for $\mathcal A_0$ is uniquely determined by $\St_0 \cap [A_0B_0]$, which is always a point.
         \item[(iv)] Let $\St = \St_1$ or $\St = \St_0$. Then for every $k \in \mathbb{N}$ there is a unique parallelogram $\mathcal{P}_k$ whose sides are parallel to the wind-rose of $\St$. 
         \item[(v)] Every full finite component of a full* Steiner tree $\St_1$ for $\mathcal A_1$ contains an odd number of terminals. 
     \end{itemize}
\end{lemma}

\begin{proof}
    Let us start with item~(iv).
    In the case of the set $\St = \St_0$ notice that the edge of $\St_0$ touching the segment 
    $[A_0B_0]$ must be parallel to the bisector and hence cannot be used 
    to construct a parallelogram $\mathcal{P}_k$. The parallelogram is hence uniquely defined by the other two directions of the wind rose.
    In the case of $\St = \St_1$ the proof is more involved. 
    Let $\mathcal C_1$ be the convex set given by 
    Lemma~\ref{lm:321066} applied to $\St_1$ with $k = 1$. 
    The Lemma assures that $A_1$ and $B_1$ are the only terminals in $\mathcal{C}_1$ and that $\St_1 \cap \partial \mathcal{C}_1$ is a point. Since $\St_1 \subset \conv \mathcal{A}_1$  (see Section~\ref{sec:steiner}), points $A_1$ and $B_1$ have degree one in $\St_1$, so $\St_1 \cap \mathcal{C}_1$ is a tripod. Thus $A_1$ and $B_1$ are adjacent to the same Steiner point, let us call it $S_1$. Consider the Melzak point $C$ for $A_1$ and $B_1$; clearly $C$ lies in the bisector. Put 
    \[
    \mathcal{A}' = \mathcal{A}_1 \cup \{C\} \setminus \{A_1,B_1\}.
    \]
    By Lemma~\ref{lm:Melzak} applied with $p_1 = A_1$, $p_2 = B_1$ and $q = C$ a full* tree $\St_1$ corresponds to a full* tree $\St'$ for $\mathcal A'$ with the same wind-rose.

    Since $\St' \subset \conv \mathcal{A}'$, the segment of $\St'$ incident to $C$ belongs to the angle $A_2CB_2$. 
    Note that $\angle A_2CB_2 < 2\alpha = \angle A_2A_\infty B_2$ because $|A_\infty A_2| = |A_\infty B_2| < |A_2 C| = |B_2C|$ since $\lambda \le \frac 1 2$. 
    Thus the wind-rose of $\St'$ contains a direction defining an angle smaller than $\alpha$ with the bisector. 
    This means that either $A_k$ or $B_k$ has only two directions from the wind-rose which go inside $\conv \mathcal{A}$ and thus $\mathcal{P}_k$ is uniquely defined for $\St_1$.

Denote the vertices of $\mathcal{P}_k$ by $A_kV_kB_kU_k$ for $k \geq 1$, where $V_k \in \conv \{A_\infty,A_k,B_k\}$. 
    We claim that the parallelograms $\mathcal{P}_k$ are disjoint. 
    Call $I$ the intersection of the segments $[A_1B_2]$ and $[A_2B_1]$. By symmetry $I$ belongs to the bisector.
    Then~\eqref{eq:438466} implies $\angle A_1IA_2 > \angle A_1W_1A_2 \geq 2\pi/3$, where $W_1$ is defined in the proof of Lemma~\ref{lm:321066}.
    So $\angle A_1IB_1 = \angle A_2IB_2 = \pi - \angle A_1IA_2 < \pi/3$. 
    Thus the domains $\{X \, | \, \angle A_1XB_1 \geq 2\pi/3 \} \supset \mathcal{P}_1$ and 
    $\{Y \, | \, \angle A_2YB_2 \geq 2\pi/3 \} \supset \mathcal{P}_2$ are separated by the perpendicular to the bisector at $I$ and hence $\mathcal{P}_1$ and $\mathcal{P}_2$ are disjoint. By similarity this implies that all parallelograms $\mathcal P_k$ are disjoint.
    
    Now we are going to prove that there is a single edge in $\St_0$ which has a point on the segment $[A_0B_0]$ which implies item~(ii). Assume the contrary, then $\St_0 \setminus [A_0B_0]$ has at least 2 components. One of them contains $A_\infty$, so its length is at least 
    \[
    \frac{\cos\alpha}{\lambda} - \frac{\sin\alpha}{\sqrt{3}\lambda}.
    \]
    Also, the length of every component cannot be smaller than the distance between $[A_0B_0]$ and $\mathcal{A}_1$ which is
    \[
    \frac{\cos\alpha}{\lambda} - \frac{\sin\alpha}{\sqrt{3}\lambda} - \cos \alpha.
    \]
    Summing up
    \[
    \H(\St_0) \geq 2\frac{\cos\alpha}{\lambda} - 2\frac{\sin\alpha}{\sqrt{3}\lambda} - \cos \alpha.
    \]
    Under the assumption~\eqref{eq:438466} this is greater than the length of the example we have\footnote{Formally we have a logical loop here, but the calculation of length in a given example is independent with any logical arguments.}, whose length is 
    \[
    \frac{\cos\alpha}{\lambda} - \frac{\sin\alpha}{\sqrt{3}\lambda} + \frac{\sqrt{3}}{1 - \lambda} \sin \alpha;
    \]
    this is equivalent to
    \begin{equation}\label{eq:appendix}
    \frac{\cos\alpha}{\lambda} - \frac{\sin\alpha}{\sqrt{3}\lambda} - \cos \alpha \geq  \frac{\sqrt{3}}{1 - \lambda} \sin \alpha
    \end{equation}
    which is shown in the end of the Appendix.

    Let us prove items~(i) and~(iii). Recall that $S_1$ is the Steiner point adjacent to $A_1$ and $B_1$ in $\St_1$, and also $S_1$ is a vertex of $\mathcal{P}_1$ (both for $\St_0$ and $\St_1$).

    Now consider, in $\St_j$, the simple path $P_k$ from $A_k$ to $B_k$ for $k > 1$ (see Fig.~\ref{pict:paralelograms}).     
    The number of connected components in $\St_j \setminus P_k$ is equal to the number of branching points in $P_k$ plus the number of terminals of degree 2 in 
    $\{A_k,B_k\}$. By Lemma~\ref{lm:321066} the path $P_k$ contains exactly two points from $\overline{\St_j \setminus P_k}$ (such points are either branching or terminal). We call these points $S_k$, $T_k$, in a way that $S_k \in [A_kV_k] \cup [V_kB_k]$ and $T_k \in [A_kU_k] \cup [U_kB_k]$.
    The corresponding edges have opposite directions in $\St_j \setminus P_k$. Thus $P_k \subset \mathcal{P}_k$.

    Reasoning by induction on $k$ the point $S_k$ determines $T_{k+1}$ and $T_{k+1}$ determines $S_{k+1}$. This proves items~(i) and~(iii).

\begin{figure}[ht]
    \centering
    \input{anglepict/4paralelograms.tex}
    \caption{Behaviour of $\St_1$ and $\St_0$ around the terminals $A_k$ and $B_k$.}
    \label{pict:paralelograms}
\end{figure}
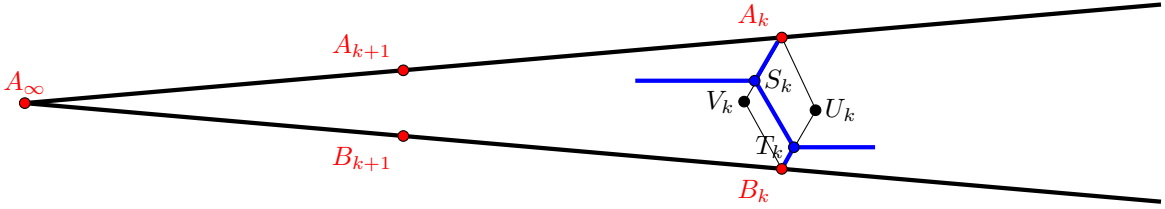

Now we proceed with item~(v). Note that in the construction we cannot have $\{S_k,T_k\} = \{A_k,B_k\}$, because otherwise we would have $P_k = [A_kB_k]$, but this edge has a direction which is not represented in the wind-rose. 
Thus the full component $\St_1'$ of $\St_1$ containing $A_1$ has an odd (or an infinite) number of terminals. Note that the terminal set of $\overline{\St_1 \setminus \St_1'}$ is similar to $\mathcal{A}$ with a scale factor $\lambda^k$, so a reduction implies item~(iii).

\end{proof}

\section{Proof of Theorem~\ref{theo:main2}} \label{sec:proof1}

\begin{lemma} \label{lm:fullstar}
     If $\alpha < \frac{\pi}{6}$, and $\lambda$ satisfy~\eqref{eq:438466}, then there exists a full* Steiner tree for $\mathcal{A}_1$.
\end{lemma}

\begin{proof}
    Let $\St_1$ be a solution to the Steiner problem for $\mathcal{A}_1$. Let $\St'$ be the full component of $\St_1$ containing $A_1$. If $\St'$ is infinite then it contains every segment between $\mathcal{C}_k$ and $\mathcal{D}_k$, defined in  Lemma~\ref{lm:321066}, so every other full component should be a segment $[A_kB_k]$ which is impossible. This means $\St' = \St_1$ and we are done.
    
    If $\St'$ is finite, then by Lemma~\ref{lm:parallelogrammes}(v) $\St'$ has an odd number of terminals and without loss of generality the largest index of a terminal of $\St'$ is reached on $A_{k+1}$.
    Then $\mathcal{A}_{k+1} := \mathcal{A} \setminus \{A_1,\dots, A_k\} \setminus \{B_1,\dots, B_k\}$ is similar to $\mathcal{A}_1$, i.e.{}
    \[
    f_k(\mathcal{A}_1) = \mathcal{A}_{k+1},
    \]
    where $f_k$ is the homothety with center $A_\infty$ and ratio $\lambda^k$.
    So 
    \[
    \H(\St_{k+1}) = \lambda^k \H(\St_1),
    \]
    where $\St_{k+1}$ is a Steiner tree for $\mathcal{A}_{k+1}$ and
    \[
    \H (\St') = (1-\lambda^k) \H (\St_1).
    \]
    Let $\St^*$ be the union of $f_k^j (\St')$ over non-negative integers $j$. Note that $\St^*$ connects $\mathcal{A}_1$ and 
    \[
    \H (\St^*) = \sum_{j=0}^\infty \lambda^{kj}\H (\St') = \frac{1}{1-\lambda^k}\H (\St') = \H (\St_1).
    \]
    Thus $\St^*$ is a solution to the Steiner problem for $\mathcal{A}_1$. 
    By construction it is full*, so the lemma is proved.
\end{proof}

\begin{lemma} \label{lm:generalcomplex}
    The length of a full* Steiner tree for $\mathcal A_1$ is equal to
    \[
    e^{-\beta i} \left ( \cos \alpha + \sqrt{3} \sin \alpha + a e^{\frac{\pi i}{6}} + b e^{-\frac{\pi i}{6}}\right ) = \left | \cos \alpha + \sqrt{3} \sin \alpha + a e^{\frac{\pi i}{6}} + b e^{-\frac{\pi i}{6}}\right |,
    \]
    where $a = 2\sin \alpha \sum_{j \in J} \lambda^j$, $b = 2\sin \alpha \sum_{j \in \mathbb{N} \setminus J} \lambda^j$ for some $J \subset \mathbb{N}$.
\end{lemma}

\begin{proof}
    Consider a complex coordinate system, determined by $A_\infty = 0$, $A_1 = \cos \alpha  + i\sin \alpha$, $B_1 = \cos \alpha  - i\sin \alpha$. Let $C$ be a Melzak point of $A_1$ and $B_1$; clearly, $C = \cos \alpha + \sqrt{3} \sin \alpha$ and
    \[
    A_k = \lambda^{k-1} (\cos \alpha  + i\sin \alpha), \quad \quad  B_k = \lambda^{k-1} (\cos \alpha - i\sin \alpha).
    \]
    Denote by $\beta$ the angle smallest angle between an edge of the tree and the real axis; by the argument in the proof of Lemma~\ref{lm:parallelogrammes}(iv), one has $\beta \in [0,\alpha]$.
    Now we apply Lemma~\ref{lm:Maxwell}. 
    By Melzak reduction we can replace $A_1$ and $B_1$ with $C$.
    The unitary coefficient at $C$ is $e^{-\beta i}$, at $A_k$ is $e^{(-2\pi/3-\beta) i}$ or $e^{(-\pi/3-\beta) i}$ and at $B_k$ is $e^{(\pi/3-\beta) i}$ or $e^{(2\pi/3-\beta) i}$, respectively. 
    By the proof of Lemma~\ref{lm:parallelogrammes}(i) the sum of the coefficients at $A_k$ and $B_k$ is zero for every $k > 1$.
    So the summands with indices $k$ are either equal to
    \[
    2\sin \alpha \, e^{(-2\pi/3-\beta) i} \lambda^{k-1} i = 2\sin \alpha \, e^{(-\pi/6-\beta) i}\lambda^{k-1}
    \] or to
    \[
    2\sin \alpha \, e^{(-\pi/3-\beta) i} \lambda^{k-1} i = 2\sin \alpha \, e^{(\pi/6-\beta) i} \lambda^{k-1}.
    \]
Proper naming of $a$ and $b$ gives us the desired result.
\end{proof}

\begin{lemma} \label{lm:morelength}
    Suppose that $\lambda \leq \frac{1}{2}$ and $\alpha < \frac \pi 6$ satisfy~\eqref{eq:438466}. 
    If $\H (\St_1^1) = \H (\St_1^2)$ for distinct full* Steiner trees $\St_1^1$ and $\St_1^2$ for $\mathcal{A}_1$, then $\St_1^2$ is the reflection of 
    $\St_1^2$ with respect to the angle bisector.   
\end{lemma}

\begin{proof}
Assume the contrary. Lemma~\ref{lm:generalcomplex} gives
\[
\left | \cos \alpha + \sqrt{3} \sin \alpha + a_1 e^{\frac{\pi i}{6}} + b_1 e^{-\frac{\pi i}{6}}\right | = \left | \cos \alpha + \sqrt{3} \sin \alpha + a_2 e^{\frac{\pi i}{6}} + b_2 e^{-\frac{\pi i}{6}}\right |.
\]
Put $c = \cos \alpha + \sqrt{3} \sin \alpha$ and $L_j = \cos \alpha + \sqrt{3} \sin \alpha + a_j e^{\frac{\pi i}{6}} + b_j e^{-\frac{\pi i}{6}}$, $j=1,2$. Then
\begin{equation} \label{eqlengthab}
    L_j \cdot \bar{L_j} = c^2 + c(a_j + b_j)(e^{\frac{\pi i}{6}} + e^{-\frac{\pi i}{6}}) + a_j^2 + b_j^2 + a_jb_j(e^{\frac{\pi i}{3}} + e^{-\frac{\pi i}{3}}) = c^2 + \sqrt{3} c (a_j + b_j) + (a_j + b_j)^2 - a_jb_j.
\end{equation}
Recall that $a_j + b_j = 2\sin\alpha \cdot  \frac{\lambda}{1-\lambda}$, so $L_1 \cdot \bar{L_1} = L_2 \cdot \bar{L_2}$ implies
$a_1b_1 = a_2b_2$, hence we have either $a_1 = a_2$ and $b_1 = b_2$ or $a_1 = b_2$ and $a_2 = b_1$.
Without loss of generality we may assume that we are in the first case. 

Let $a_1$ and $a_2$ be the sums of $\lambda^k$ over $k \in J_1$ and $k \in J_2$, respectively.
Since $\lambda \leq \frac{1}{2}$ the sets of indices $J_1$ and $J_2$ coincide, say $J_1 = J_2 := J$.
Define an abstract graph $G = (V,E)$ with vertices
\[
V := \{A_k,B_k,S_k,T_{k+1} \, | \, k \in \mathbb{N}\}
\]
and edges 
\[
E := \{(A_1,S_1), (S_1,B_1)\} \cup \{(S_k,T_{k+1}) \, | \, k\in \mathbb{N} \} \cup 
\]
\[
\{(A_{k+1},S_{k+1}), (B_{k+1},T_{k+1}) \, | \, k \in J \} \cup \{(A_{k+1},T_{k+1}), (B_{k+1},S_{k+1}) \, | \, k \in \mathbb{N} \setminus J \}.
\]
Recall that $\St_1^1$ and $\St_1^2$ are Steiner trees, so they are global (and hence local) minima of $\ell_G(\cdot)$. 
By Theorem~\ref{th:uniq} for $G$ and $V_0 = \mathcal{A}_1 \setminus \{A_\infty\}$ the trees $\St_1^1$ and $\St_1^2$ have a common wind-rose and by Lemma~\ref{lm:parallelogrammes}(i) they coincide. Thus in the case $a_1 = b_2$ and $a_2 = b_1$ the trees $\St_1^1$ and $\St_1^2$ are symmetric.
\end{proof}

Now we are ready to prove Theorem~\ref{theo:main2}.

\begin{remark} \label{rem:localminimas}
    Here we have shown that a full* Steiner tree for $\mathcal{A}_1$ is determined (up to the reflection) by knowing which of the two possible directions of the 
    wind-rose is used in each vertex $A_k$ (this determines the set $J\subset \mathbb N$). 
    We don't know which are the sets $J$ that correspond to a locally minimal tree: the answer may depend on $\alpha$ and $\lambda$ in some complicated way (we discuss it in Section~\ref{sec:dynamics}). 
    Fortunately, a relatively simple argument allows us to find a very strict 
    condition on $J$ which must hold for an absolute minimimer.
\end{remark}

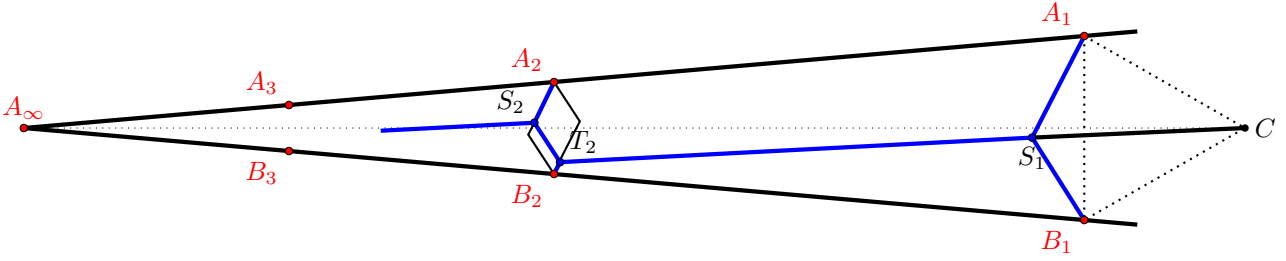
\begin{figure}[h]
    \centering
    \input{anglepict/6observations.tex}
    \caption{The construction used in the proof of Theorem~\ref{theo:main2}.}
    \label{pict:observations}
\end{figure}

Consider the complex coordinate system from Lemma~\ref{lm:generalcomplex}.

Let $\St_1$ be a full* Steiner tree for $\mathcal A_1$ and let $\mathcal{P}_k$ be the parallelogram $A_kV_kB_kU_k$ defined in Lemma~\ref{lm:parallelogrammes}(iv). Without loss of generality suppose that $S_1 = V_1$ lies in the lower half-plane (otherwise consider the reflection of $\St_1$). 
Thus every $U_j$ lies in the upper half-plane and every $V_j$ belongs to the lower half-plane. 
Clearly for every $j$ the segments $[U_jA_j]$ lie in the upper half-plane and $[V_jB_j]$ lie in the lower half-plane.
We need the following observation.

\begin{observation} \label{obs:obs}
\begin{enumerate}
    \item[(i)] One has $|V_jB_j|=|U_jA_j|<|U_jB_j|=|A_jV_j|$ for every $j$ as $U_j$ belongs to the upper half-plane and the inequality $|XA_j| < |XB_j|$ holds for every point $X$ from the upper half-plane.
    \item[(ii)] By Lemma~\ref{lm:Melzak} the point $T_2$ lies in the lower half-plane as $T_2$ belongs to the ray $[C,S_1)$, where $C$ is the Melzak point of $A_1$ and $B_1$ (recall that $C$ belongs to the real axis) and thus $S_1$ belongs to the lower half-plane. Then we have $T_2 \in [V_2B_2]$.  
    \item[(iii)] If $T_j$ belongs to the upper half-plane then $T_{j+1}$ belongs to the lower half-plane as the ray $[S_jT_{j+1})$ is co-directed with the ray $[S_1T_2) \subset [CT_2)$ and $S_j$ belongs to $[V_jB_j]$ and hence to the lower half-plane. 
\end{enumerate}
\end{observation}

Lemma~\ref{lm:generalcomplex} says that
    \[
    \H(\St_1) =  e^{-\beta i} \left ( \cos \alpha + \sqrt{3} \sin \alpha + a e^{\frac{\pi i}{6}} + b e^{-\frac{\pi i}{6}}\right ) = \left | \cos \alpha + \sqrt{3} \sin \alpha + a e^{\frac{\pi i}{6}} + b e^{-\frac{\pi i}{6}}\right |,
    \]
where $a = 2\sin \alpha \sum_{j \in J} \lambda^j$, $b = 2\sin \alpha \sum_{j \in \mathbb{N} \setminus J} \lambda^j$ for some $J \subset \mathbb{N}$.
Looking at equation~\eqref{eqlengthab} one notices 
that since $a+b$ is fixed (independent of $J$) the total length is smaller when $ab$ is larger, or, which is the same, when $|a-b|$ is smaller. 
By Observation~\ref{obs:obs}(ii), one has $1 \in J$. Observation~\ref{obs:obs}(iii) says that $j$ and $j+1$ cannot belong to $\mathbb{N} \setminus J$ simultaneously.
So $|a-b|$ is minimal for $J = \{2k-1 \, | \, k\in \mathbb{N}\}$. 
The construction described in the proof of Lemma~\ref{lm:fullstar} started with a full tree for $A_1,A_2,A_3,B_1,B_2$ shows that such $J$ is admissible. So the following lemma concludes the proof.

\begin{lemma} \label{lm:5points} 
Suppose that $\lambda \leq \frac{1}{2}$ and $\alpha < \pi/6$ satisfy the condition~\eqref{eq:438466}. 
Then every Steiner tree for $\mathcal{A}_5 := \{A_1,A_2,A_3,B_1,B_2\}$ is full.
\end{lemma}
\begin{proof}
Assume the contrary, then a Steiner tree $\St_5$ contains a terminal $X$ of degree 2. Recall that the Steiner tree is contained in $\conv \mathcal{A}_5 = A_3A_1B_1B_2$ (see Figure~\ref{pict:mainth2}). Clearly, $\angle A_3A_1B_1 =  \angle A_1B_1B_2 = \pi/2 - \alpha < 2\pi/3$ and $\angle B_2A_3A_1 < \angle B_3A_3A_1 = \pi/2 + \alpha < 2\pi/3$.
Thus $X$ is either $A_2$ or $B_2$. Repeating the argument of Lemma~\ref{lm:321066} we conclude that $A_1$ and $B_1$ are adjacent to the same Steiner point $Y$. Now if both $A_2$ and $B_2$ have degree 2, then $\St_5$ contains segment $[A_2B_2]$, which is impossible since 
$\angle A_3A_2B_2, \angle A_3B_2A_2 \leq \pi/2 - \alpha$. 

Now suppose that $\St_5$ contains a full component which is an edge $[XZ]$. Clearly $[XZ]$ cannot split $\conv \mathcal{A}_5$ because $[XZ]\neq[A_2B_2]$, 
so $Z = A_3$. Since $|A_2A_3| < |B_2A_3|$, in this case $\St_5$ is a union of $[A_2A_3]$ and a symmetric solution for $\{A_1,B_1,A_2,B_2\}$.
But then $\angle A_3A_2W = 2\pi/3 - \alpha$, where $W$ is the second Steiner point, a contradiction.

In the remaining case $\St_5$ is a union of two similar tripods, which is longer by~\eqref{eqlengthab}.
\end{proof}

\section{Proof of Theorem~\ref{theo:main}}  \label{sec:proof2}

In this section, we prove Theorem~\ref{theo:main}. Clearly $\mathcal{A}_0$ is a compact set, so the Steiner problem has a solution.
By Lemma~\ref{lm:parallelogrammes}(ii) every Steiner set for $\mathcal{A}_0$ is a Steiner tree. 
Suppose that $\St_0$ is a Steiner tree for $\mathcal{A}_0$; obviously $\H(\St_0) < \infty$.
Denote the point $\overline{\St_0} \cap [A_0B_0]$ by $x$. 

\begin{lemma}\label{lm:nedofull}
    $\St_0$ is a full* tree.
\end{lemma}

\begin{proof}
    Consider the full* subtree of $\St_0$ containing $x$
    which is maximal with respect to inclusion. Call it $\St_0'$.
    Since the segment of $\St_0$ containing $x$ is orthogonal to $[A_0B_0]$, the wind rose of $\St_0'$ contains the direction of the bisector of the angle. 
    Suppose, by contradiction that $\St_0 \setminus \St_0'$ is non-empty.  Then $\St_0'$ contains a terminal (without loss of generality suppose it is $A_k$) which has 
    degree~2 in $\St_0$ and degree 1 in $\St_0'$. 

    Let $[A_ky]$, $[A_kz]$ be the edges incident to $A_k$, where $z$ is a vertex of $\St_0'$ and $y$ is a vertex of $\St_0$ but not of $\St_0'$.
    Since $\St_0'$ is maximal $\angle yA_kz > 2\pi/3$ and hence the segment 
    $[A_kz]$ must be parallel to the bisector.
    
    If $y$ is a terminal point then by Lemma~\ref{lm:321066} $y \in \{ A_{k+1},B_{k+1}\}$. 
    By Lemma~\ref{lm:321066} the edge $[A_ky]$ contains the segment of $\St_0$ connecting $\mathcal{C}_k$ and $\mathcal{D}_k$. Hence $A_k$ and $B_k$ belong to the same component of $\St_0 \setminus [yA_k[$. By Lemma~\ref{lm:321066} the tree $\St_0$ contains unique segment connecting $\mathcal{C}_{k-1}$ and $\mathcal{D}_{k-1}$, so the path connecting $A_k$ and $B_k$ in $\St_0$ has exactly one Steiner point, which should be $z$. This is impossible since $[A_kz]$ is parallel to the bisector.
    
    Thus $y$ is a branching point and so $[yB_k] \subset \St_0$; recall that the angle $yA_kz$ is strictly greater than $2\pi/3$. Let $\St$ be the union of $\St_0 \setminus \St_0'$ and the reflection of $\St_0'$. Let $z'$ be the reflection of $z$, then $B_kz'$ is also parallel to the bisector.
    Note that $\H(\St) = \H(\St_0)$, so $\St$ is also a solution to the Steiner problem for $\mathcal{A}$. But 
    \[
    \angle yB_kz' = 2\pi - \angle zA_ky - \angle A_kyB_k < 2\pi/3,
    \]
    so $\St$ is not a Steiner tree, contradiction.
    
\end{proof}

\begin{lemma}\label{lm:full}
    $\St_0$ is a full tree.
\end{lemma}

\begin{proof}
By Lemma~\ref{lm:nedofull} $\St_0$ is a full* tree. 
Let $\St_0'$ be the full component containing $x$. Consider a terminal of $\St_0'$ with degree $2$ (without loss of generality it is $A_k$). 
Then the structure of the other full component $\St$, containing $A_k$ is uniquely determined 
$S_k$ is the vertex of the rhombus $\mathcal{R}_k$. Thus $\St$ is a Steiner tree for 4 terminals $A_k, B_k, A_{k+1}, B_{k+1}$ (and its branching points are $S_k, T_{k+1}$). 
But 
\[
\angle A_\infty A_{k+1} T_{k+1} = \angle A_\infty B_{k+1} T_{k+1} = 2\pi/3 - \alpha
\]
which means that both $A_{k+1}$ and $B_{k+1}$ cannot be a point of degree 2 in $\St_0$. This is a contradiction.
\end{proof}

\begin{lemma} \label{lm:alternating}
  Consecutive horizontal parts of $\St_0$ lie on the opposite sides of the bisector of the angle.
\end{lemma}

\begin{proof}
    We repeat the proof of Lemma~\ref{lm:parallelogrammes}(iii). Every $\mathcal{P}_i$ is a rhombus with two vertices at the bisector of the angle.
    By Lemma~\ref{lm:full} no segment of $\St_0$ belongs to the bisector.
    Thus $S_{i+1}$ and $T_{i+1}$ always belong to the opposite sides of the bisector. Since $S_{i}T_{i+1}$ is parallel to the bisector for every $i$, we are done. 
\end{proof}

\begin{lemma} \label{lm:complexlength}
Consider a complex coordinate system, such that $A_\infty = 0$, $A_1 = \cos \alpha  + i\sin \alpha$, $B_1 = \cos \alpha  - i\sin \alpha$. 
Let $x$ be the point where $\overline{\St_0}$ touches $[A_0B_0]$
and suppose that $\lvert A_0x\rvert  \leq \lvert B_0x\rvert$.
Then
    \[
    \H(\St_0) = 
    x + \left( \frac{\sqrt 3}{1 - \lambda} - \frac{i}{1 + \lambda} \right) \sin \alpha
    \]
    where $x$ is the complex number representing the point $x$.
In the case when $\lvert A_0x\rvert  \leq \lvert B_0x\rvert$ we obtain 
the same formula with $x$ replaced by $\bar x$.
\end{lemma}

\begin{proof}
    We are going to use the length formula~\eqref{eq:maxwell}:
    \begin{equation}
        \label{eq:ckdk}
      \H(\St_0) 
        = \sum_{j=1}^{\infty} \bar{c_j} A_j 
        + \sum_{j=1}^{\infty} \bar{d_j} B_j + 1\cdot x.   
    \end{equation}
    By Lemma~\ref{lm:full} every terminal has degree 1. 
    If we suppose that $\lvert A_0x\rvert  \leq \lvert B_0x\rvert$ 
    the first horizontal edge of $\St_0$ starting from $x$ is going 
    to hit the upper part of the first rhombus $\mathcal{R}_1$.
    This means that the first Steiner point $S_1$ is connected to the upper 
    terminal $A_1$ and has a direction $e^{2\pi i/3}$. 
    The edge connected to the terminal $B_1$ has the opposite direction.
    By Lemma~\ref{lm:alternating} the next horizontal edge is below the angle 
    bisector and the direction of the edges on the terminals $A_2$ and $B_2$
    are rotated by $-\pi/3$. The angles keep alternating between even and 
    odd vertices,
    so, one has the following coefficients $c_k,d_k$ in~\eqref{eq:ckdk}: 
    $c_{2k-1} = e^{2\pi i/3}$, 
    $c_{2k} = e^{\pi i/3}$, 
    $d_{2k-1} = e^{-\pi i/3}$ and 
    $d_{2k} = e^{-2\pi i/3}$, $k \in \mathbb{N}$. 
    The coefficient for $x$ is $1$.
    Recall that $A_j=\lambda^j e^{i\alpha}$, $B_j=\lambda^j e^{-i\alpha}$.
    For odd $j$ one has $\overline{c_j} A_j + \overline{d_j} B_j 
    = 2\lambda^j \sin \alpha \cdot i \cdot e^{-2\pi i/3 }$ while for 
    even $j$ one obtains $2\lambda^j \sin \alpha \cdot i \cdot e^{- \pi i/3 }$. 
    Thus \eqref{eq:ckdk} becomes:
    \begin{align*}
    \H (\St_0) 
    &= x + 2 (1 + \lambda^2 + \lambda^4 \dots ) i e^{-2\pi i/3 } \sin \alpha + 2\lambda (1 + \lambda^2 + \lambda^4 + \dots ) i e^{- \pi i/3 } \sin \alpha\\
    &= 
    x + \left( \frac{2 e^{-2\pi i /3} }{1 - \lambda^2} + \frac{2\lambda e^{-\pi i /3} }{1 - \lambda^2}  \right) i \sin \alpha =
    x + \left( \frac{2 e^{- \pi i /6} }{1 - \lambda^2} + \frac{2\lambda e^{ \pi i /6} }{1 - \lambda^2}  \right) \sin \alpha \\
    &= x + \left( \frac{\sqrt 3}{1 - \lambda} - \frac{i}{1 + \lambda} \right) \sin \alpha.
    \end{align*}
    
\end{proof}

Now we are ready to finish the proof of the theorem. By Lemma~\ref{lm:full} $\St_0$ is full*. 
Lemma~\ref{lm:complexlength} gives a formula for the length, which should be a real number.
Thus
\[
\Im x  = \pm \frac{1}{1 + \lambda} \sin \alpha.
\]
This determines the point $x$ proving item~(i).
Recall that $\Re x = \frac{\cos\alpha}{\lambda} - \frac{\sin\alpha}{\sqrt{3}\lambda}$, so 
\[
\H(\St_0) = \frac{\cos\alpha}{\lambda} - \frac{\sin\alpha}{\sqrt{3}\lambda} + \frac{\sqrt{3}}{1 - \lambda} \sin \alpha
\]
and item~(ii) is also proven.

We are left with the proof of item~(iii).
Clearly $\St_0' = \St_0 \cap \conv f_2(\mathcal A_0)$ connects $f_2(\mathcal A_0)$.
Note that $\St_0'$ is a full tree so we may apply Lemma~\ref{lm:Maxwell} to compute $\H(\St_0')$.
This computation is essentially contained in the computation in the proof of Lemma~\ref{lm:complexlength}; the only difference is that we omit two starting terms in geometric progressions. Thus we obtain
\[
\H \left (\St'_0 \right ) = \lambda^2 \H (\St_0).
\]
It means that $\St_0 \cap f_2([A_0B_0])$ is a Steiner tree for $f_2(\mathcal A_0)$ and Lemma~\ref{lm:alternating} gives 
$f_2(x) = \St_0 \cap f_2([A_0B_0])$.
Also $f_{-1} (\mathcal{P}_1) \cap [A_0B_0] = f_{-1}(V_1)$ so $f_{-1} (\mathcal{P}_1)$ does not intersect $\St_0$. 
Thus $f_2(\St_0)$ does not intersect $\mathcal{P}_2$ and so $f_2(\St_0) \subset \St_0$.

\section{Piecewise contraction dynamics}  \label{sec:dynamics}

Lemma~\ref{lm:parallelogrammes}(i) assures that any full* tree with terminals 
on the set $\mathcal A_1$ is uniquely determined by its wind-rose.
If one knows the position of the line containing the long edge joining the first two parallelograms $\mathcal P_1$ and $\mathcal P_2$ then the whole tree is determined. 
Since the points of $\mathcal A_1$ are placed in geometric progression, each 
subsequent step of the tree construction is subject to the same recurrence equation with the appropriate rescaling. 

It turns out that the resulting discrete dynamical system is conjugate to
the inverse of a 2-interval piecewise contraction, whose 
law can be written as
\[
F_{\lambda,\delta} : x \in [0,1] \to \{ \lambda x + \delta \},
\]
where $\{\cdot\}$ stands for the fractional part.
See \cite{laurent2018rotation}, \cite{pires2019symbolic}, 
\cite{gaivao2022dynamics}.

Suppose $\St$ is a full* tree satisfying the assumptions of Lemmas~\ref{lm:321066} and~\ref{lm:parallelogrammes}. 
We know that the segments of $\St$ outside the parallelograms $\mathcal P_i$ are all parallel to each other. 
Let $e_1$ be the unit vector parallel to these segments 
pointing from $T_2$ towards $S_1$ in Figure~\ref{pict:observations}.
Let $e_2$ and $e_3$ be the two unit vectors completing the wind-rose of $\St$ so that
$e_1$, $e_2$, $e_3$ have zero sum and are counter-clockwise oriented ($e_2$ is pointing from $T_2$ towards $S_2$ and $e_3$ is pointing from $T_2$ towards $B_2$).
It is convenient to use a hexagonal (barycentric) coordinate system centered in $A_\infty$.
A point $p$ is an equivalent class of triples $(u,v,w)$ under the relation $(u,v,w) = (u+t,v+t,w+t)$ which corresponds to the representation $p = u e_1 + v e_2 + w e_3$.
Usually we will choose the representant with $v + w = 0$. 

Using the notation of the proof of Lemma~\ref{lm:parallelogrammes}
let $(l,\delta,-\delta)$ be the hexagonal coordinates of the point $C_1$ which is the center of the parallelogram $\mathcal{P}_1$. 
In general we have $\delta \neq 0$ because the direction $e_1$ is not parallel to the angle bisector.
Let $2a := |A_1V_1| = |U_1B_1|$ and $2b := |A_1U_1| = |V_1B_1|$.
One has:
\begin{align*}
A_1 &= C_1+(0,b,-a) = (l,b+\delta,-a-\delta) 
     = \left (l + \frac{a-b}{2}, \delta + \frac{a+b}{2}, -\delta- \frac{a+b}{2} \right ), 
\\
B_1 &= C_1+(0,-b,a) = (l,\delta-b,a-\delta) 
     = \left(l - \frac{a-b}{2},  \delta - \frac{a+b}{2},  \frac{a+b}{2} - \delta \right),
\\ 
U_1 &= C_1+(0,-b,-a) = (l,\delta-b,-a-\delta)
     = \left (l + \frac{a+b}{2} , \delta + \frac{a-b}{2}, - \delta - \frac{a-b}{2} \right ), 
\\
V_1 &= C_1+(0,b,a) = (l,\delta+b,a-\delta)
    = \left(l - \frac{a+b}{2}, \delta - \frac{a-b}{2},  \frac{a-b}{2} - \delta \right).
\end{align*}

Then define $C_k = \lambda^{k-1}C_1$, $A_k = \lambda^{k-1}A_1$, $B_k = \lambda^{k-1}B_1$, $U_k = \lambda^{k-1}U_1$ and $V_k = \lambda^{k-1}V_1$.

The line segments $[S_{k-1} T_k]$ connecting the parallelograms $\mathcal P_{k-1}$ and $\mathcal P_k$ identified by Lemma~\ref{lm:321066} are parallel to the direction $e_1$ hence have coordinates of the form $(\cdot,\mu_{k-1},-\mu_{k-1})$. 
We are going to find a recurrence relation between $\mu_{k-1}$ and $\mu_k$.
Lemma~\ref{lm:parallelogrammes}(i) gives a rule to obtain $\mu_k$ from $\mu_{k-1}$, or to show that there is no corresponding $\mu_k$. 
Since the prolongation of the segment between the parallelograms $\mathcal P_{k-1}$ and $\mathcal P_k$ must pass between the points 
$A_k$ and $B_k$, one should have 
\[
    \left(\delta -\frac{a+b}{2}\right) \lambda^{k-1}
    \le \mu_{k-1} 
    \le \left(\delta + \frac{a+b}{2} \right)\lambda^{k-1}.
\]
We consider a renormalized version of $\mu_k$ by defining 
\begin{equation}\label{eq:normalization}
    \nu_k := \frac 1 2 + \frac{\lambda^{-k}\mu_k - \delta}{a+b},
    \qquad \mbox{i.e.}\qquad 
    \mu_k = \lambda^{k}\left((a+b)\left(\nu_k-\frac 1 2 \right)+\delta\right)
\end{equation}
so that the previous conditions on $\mu_k$ become $\nu_k \in [0,1]$.

Suppose, without loss of generality, that $b > a$. 
Define
\begin{align*}
    q^\pm &:= \frac 1 2 +\frac {(1-\lambda)\delta -a} {\lambda(a+b)} \pm \frac 1 {2\lambda}\\
    t_1 &:= \lambda (1-q^+),\quad
    t^* := \frac{a}{a+b},\quad
    t_2 := -\lambda q^-.
\end{align*}

\begin{lemma} \label{lm:dynamics}
One has $\nu_{k+1} = g(\nu_k)$ with
\begin{equation}\label{eq:recurrence_g}
    g(t) = \begin{cases}
        \frac{t}{\lambda} + q^+ & \text{if } 0 \le t < t^*,\\
        \frac{t}{\lambda} + q^- & \text{if } t^* < t \le 1.
    \end{cases}
\end{equation}
Moreover $g([0,t_1]\cup [t_2,1]) = [0,1]$, $g(0)=g(1)$, $g$ injective on $\left]0,1\right[\setminus\{t^*\}$. 
If we restrict the values of $g$ to the interval $[0,1]$,
the inverse function,
$g^{-1}\colon [0,1]\setminus\{q^+\} \to[0,1]$ is uniquely defined and one has 
\[
  g^{-1}(t) = \left\{\lambda t + t_2\right\}
\]
where $\left\{x\right\}=x-\lfloor x\rfloor$ is the fractional part of $x$.
\end{lemma}

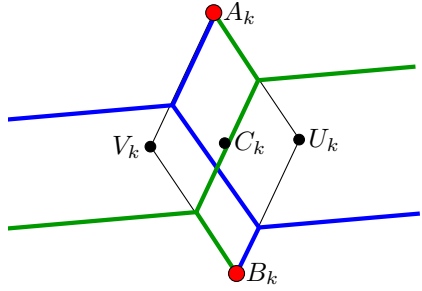
\begin{figure}[h]
    \centering
    \input{anglepict/7dynamics.tex}
    \caption{Possible behaviour of a locally minimal tree in 
    the parallelogram $\mathcal{P}_k$.}
    \label{pict:dynamics}
\end{figure}

\begin{proof}
We have only two possible behaviours at $\mathcal{P}_k$ (they are drawn in blue and green in Fig.~\ref{pict:dynamics}).
When the input line coming from the right hand side hits the segment $[A_k,U_k]$ (green case)
the line goes down, on the opposite side $[B_k,V_k]$. 
The output line is obtained adding the vector $2a \lambda^{k} e_3$ 
to the input. In hexagonal coordinates:
\[
  (t,\mu_k,-\mu_k) 
  \mapsto (t, \mu_k,2a\lambda^{k} - \mu_k) 
  = (t-a\lambda^{k}, \mu_k - a\lambda^{k}, a \lambda^{k} -\mu_k).
\]
In this case we obtain $\mu_{k+1} = \mu_k - a \lambda^{k}$.
This happens when the line $(t,\mu_k,-\mu_k)$ passes above $U_k=\lambda^{k}U_1$, that is, when $\mu_k > \lambda^{k}\left(\delta + \frac{a-b}{2}\right)$.

Otherwise, if $\mu_k<\lambda^{k}\left(\delta + \frac{a-b}{2}\right)$, 
the input line hits the segment $[B_k,U_k]$ and it goes up, 
towards the opposite side $[A_k,V_k]$.
In this case we add the vector $2 b \lambda^{k} e_2$ obtaining:
\[
  (t, \mu_k,-\mu_k) 
    \mapsto (t, \mu_k+2b \lambda^{k},-\mu_k) 
    = (t-b \lambda^{k}, \mu_k+b \lambda^{k},-\mu_k-b \lambda^{k}).
\]
Hence $\mu_{k+1} = \mu_k + b \lambda^{k}$.

A line hitting the extremal points $A_k$ or $B_k$ will cause the tree to 
have a terminal of degree two (a finite full subtree). 
A line hitting the point $U_k$ would split hitting both terminals 
$A_k$ and $B_k$: by the argument of Lemma~\ref{lm:parallelogrammes}(v) this is not possible in a minimal Steiner tree.

In conclusion, we have $\lambda^{-1-k} \mu_{k+1} = h(\lambda^{-k}\mu_k)$ with
\[
  h(x) = \begin{cases}
      \frac{x+b}{\lambda} &\text{if $-\frac{a+b}{2}<x-\delta <\frac{a-b} 2$}\\\\
      \frac{x-a}{\lambda} &\text{if $\frac{a-b}{2}<x -\delta < \frac{a+b}{2}$}.
  \end{cases}
\]
To find the renormalized recurrence relation 
we put $\nu_k=\psi(\lambda^{-k}\mu_k)$ with
\[
  \psi(s) = \frac 1 2 +\frac{s-\delta}{a+b}, \qquad
  \psi^{-1}(t) = (2t - 1)\frac{a+b}{2}+\delta
\]
whence $\nu_{k+1}=g(\nu_k)$ with $g = \psi \circ h \circ \psi^{-1}$. 
With a straightforward computation one finds:
\[
g(t) = \psi(h(\psi^{-1}(t))) = \frac{t}{\lambda} + \frac 1 2 + \frac{(1-\lambda)\delta-a}{\lambda(a+b)}\pm \frac{1}{2\lambda}
\qquad \text{if $t \lessgtr \frac{a}{a+b}$}.
\]
This gives \eqref{eq:recurrence_g}.

One has $g(t)=\frac{t}{\lambda} + q^\pm$. Hence 
the inverse function has the form $g^{-1}(s) = \lambda(s-q^\pm)$.
For $q^+<s\le 1$ one has $g^{-1}(s)=\lambda(s-q^+)$.
Since $0<\lambda<1$ it is clear that in this case $0<\lambda(s-q^+)<1$ hence $g^{-1}(s)=\{\lambda(s-q^+)\}$ for $q^+<s\le 1$.

If $0\le s < q^+$ one has $g^{-1}(s) = \lambda(s-q^-)$.
But notice that $q^- = q^+-\frac 1 \lambda$ hence 
$\lambda(s-q^-) = \lambda(s-q^+)+1$.
Notice that in this case $-1<\lambda(s-q^^+)<0$
and hence, also in this case, $g^{-1}(s) = \lambda(s-q^+)+1 = \{\lambda(s-q^+)\}$.

\end{proof}

Now consider the dynamical system $([0,1], g)$ which depends only on $\beta$. 
First, recall our proof scheme for $\mathcal{A} = \mathcal{A}_1$ (Theorem~\ref{theo:main2}). 
We restrict the attention on full* candidates, then we show that the length is determined by $J$ 
and the simplest cut of cases gave us that an optimal $J$ is $2\mathbb{N}$. 

Suppose now that one wants to check an existence of a competitor for a given $J$. 
Then $\beta$ can be easily determined as the argument in Lemma~\ref{lm:generalcomplex}.
Lemma~\ref{lm:dynamics} explicitly defines the dynamics and it remains to check if the trajectory started in $\nu_1$ is infinite.
For a full* solution from Theorem~\ref{theo:main2} the corresponding trajectory is 2-periodic with the period $\nu_1,0$ or $\nu_1,1$.
Note that this gives a way to finish the proof of Theorem~\ref{theo:main2} avoiding Lemma~\ref{lm:5points}.

Also, let us note that the trajectory which corresponds to a locally minimal full tree for $\mathcal{A}_1$ is not periodic. 

One can force a full* tree to have a given $\beta$ by addition a segment $[A_\beta B_\beta]$ with a proper slope to $\mathcal{A}_1$.
In the particular case of $\beta = 0$ we add the segment $[A_0B_0]$ to obtain the input $\mathcal{A}_0$; other parameters are $a=b$, $\delta = 0$ and $g^{-1}(t) = \{\lambda t + 1-\lambda/2\}$.
Steiner trees $\St_0^1$ and $\St_0^2$ from Theorem~\ref{theo:main} again correspond to the 2-periodic points of $([0,1],g)$, namely $t = \frac{\lambda}{2(\lambda+1)}$ and $t = \frac{\lambda+2}{2(\lambda+1)}$. However for a general $\beta$ there is no simple analogue of Lemma~\ref{lm:alternating}, which again makes the situation more intricate.

\section{Open questions} \label{sec:open}

In this section, we bring to the reader's attention both long-standing open problems and new questions we found the most interesting. 

\paragraph{Gilbert--Pollak conjecture on the Steiner ratio.} 
A minimum spanning tree is a shortest connection of a given set $A$ of points in the plane by segments with endpoints in $A$.
It is well known that a minimum spanning tree can be found in a polynomial time.

The Steiner ratio is the infimum of the ratio of the total length of Steiner tree to the minimum spanning tree for a finite set of points in the Euclidean plane.
The most known open problem on Steiner trees is Gilbert--Pollak conjecture~\cite{gilbert1968steiner} which states that the planar Steiner ratio is reached on the vertex set of an equilateral triangle (and so is equal to $\frac{\sqrt{3}}{2}$). Despite the fact that some ``proofs'' are published this conjecture remains open, see~\cite{ivanov2012steiner}. 
Graham~\cite{graham2019some} offered 1000\$ for a positive resolution.

\paragraph{Analogues results in space.} Steiner trees in a Euclidean $d$-dimensional space have the same local structure as in the plane, id est every branching point has degree 3 and moreover three adjacent segments have pairwise angles $2\pi/3$ and are coplanar.
However, a global structure is much more complicated. In particular a full (and full*) tree may have no wind-rose.
Up to our knowledge all known series of explicit solutions are planar.

In particular, we are interested in a solution to the Steiner problem for an input which is a mix of a regular $n$-gon with the setup of this paper. 
Consider a regular $3$-dimensional $n$-gon in the plane $x = 1$ centered at $(1,0,0)$; let $\mathcal{Q}$ be the set of its vertices.
Then the input is the union of all $f^k(\mathcal{Q})$, $k \in \mathbb{N}$, where $f$ is a homothety with the center in the origin and a small enough scale factor $\lambda$.

Some of our tools (in particular, Lemma~\ref{lm:473988}) can be translated to the $d$-dimensional setup.

\paragraph{The number of solutions.} 
It is known~\cite{gilbert1968steiner} that the total number of full Steiner topologies (the underlying abstract graphs) with $n$ terminals is 
\[
\frac{(2n-4)!}{2^{n-2}(n-2)!}.
\]
How many of them may simultaneously give a shortest solution to the Steiner problem? 
In all known examples (for instance, the set of ladders or the set of trees from Corollary~\ref{cor:finite}) the number is at most exponential in $n$ while the number of topologies has a factorial order of growth. Note that the answer may heavily depend on the dimension.

On the other hand, for $n \geq 4$ the set of planar $n$-point configurations with multiple solution to Steiner problem has Hausdorff dimension $2n - 1$ (as a subset of $(\mathbb{R}^2)^n$), see~\cite{basok2018uniqueness}. 
It is interesting to find the corresponding dimension in a $d$-dimensional Euclidean space.

\paragraph{Density questions.}
The following questions make sense in both two-dimensional Euclidean space and in a higher dimension.
Some of them seem not very difficult and can be considered as exercises. 
Let $\St$ be a Steiner tree and suppose that $B_r(x)$ (or another good body) contains no terminal point of $\St$. 
What is the largest possible number of branching points of $\St$ in $B_r(x)$? What is the maximum of $\H(\St \cap B_r(x))/r$? 

Notice that in this paper, we provided an example of a full Steiner tree with terminals on the boundary of a triangle and with infinitely many Steiner points in the interior (see Remark~\ref{rem:main}).

Let $\mathcal{A}$ be a finite subset of $\partial B_1(0)$ (or another good set). What is the largest possible $\sharp(\St \cap \partial B_{1/2}(0))$? 
What is the largest possible number of branching points of $\St$ lying within $B_{1/2}(0)$?
What is the upper bound for $\H(\St \cap B_{1/2}(0))$?
Similar questions can be asked for any radius in place of $\frac 1 2$.

\section{Statements \& Declarations}

\paragraph{Acknowledgments.} We thank Pavel Prozorov for finding some gaps in the initial proof. 
We are grateful to Alexey Gordeev for some computer simulations and to Carlo Carminati for bringing the 2-interval piecewise contraction to our attention.

\paragraph{Funding.} Yana Teplitskaya is supported by the Simons Foundation grant 601941, GD.
Emanuele Paolini was supported by the projects: PRIN 2022PJ9EFL and PRA GeoDom.

\paragraph{Competing Interests.} The authors have no relevant financial or non-financial interests to disclose.

\paragraph{Data Availability.} The manuscript has no associated data.

\section*{Appendix: proofs of auxiliary and well-known results}

\begin{proof}[Proof of Lemma~\ref{lm:isolated}]
    If $\mathcal A = \{x_0\}$ the result is trivial. 
    Without loss of generality, we can hence suppose 
    that $\mathcal A$ contains at least two points.
    
    \emph{Claim 1: there exists $\varepsilon>0$ such that every connected component 
    of $\St\setminus \mathcal A$ which intersect $\overline{B_\varepsilon(x_0)}$ contains $x_0$.}
    Let $\rho>0$ be such that $\overline{B_{2\rho}(x_0)} \cap \mathcal A = \{x_0\}$.
    If a connected component $S$ of $\St \setminus \mathcal A$ touches $\overline{B_\rho(x_0)}$
    but does not contain $x_0$, then it must touch both $\partial B_\rho(x_0)$ and 
    $\partial B_{2\rho}(x_0)$. 
    Being $S$ connected we would have $\H(S)\ge \rho$. 
    Since $\H(\St)<+\infty$,
    we conclude that only a finite number of connected components of $\St\setminus \mathcal A$
    can possibly intersect $\overline{B_\rho(x_0)}$ without containing $x_0$. 
    Take $\varepsilon$ smaller than the minimum distance of $x_0$ from each of these components to 
    get the claim.

    \emph{Claim 2: let $S_0$ be the connected component of $\overline{\St}$ which contains $x_0$. 
    Then $\overline{\St \cap B_\varepsilon(x_0)} =  S_0 \cap \overline{B_\varepsilon(x_0)}$}.
    Suppose by contradiction that there exists $x\in \overline{\St \cap B_\varepsilon(x_0)} \setminus S_0$.
    Let $S$ be the connected component of $\overline{\St}$ which contains $x$
    and let $S'$ be the connected component of $\St \setminus \mathcal A$ which contains $x$.
    Clearly $S'\subset S$ and $x\in S' \cap \overline{B_\varepsilon(x_0)}$.
    This is in contradiction with Claim 1.

    \emph{Claim 3: $S_0\cap \partial B_t(x_0)\neq \emptyset$ for all $t<\varepsilon$.}
    Suppose by contradiction that $\partial B_t(x_0)\cap S_0 = \emptyset$ for some $t<\varepsilon$. 
    Then $S_0 \subset B_t(x_0)$ because $x_0\in S_0$ and $S_0$ is connected.
    This means that $S_0$ is also a connected component of $\St\cup \mathcal A$
    and this is not possible because $\St\cup \mathcal A$ is connected by assumption and 
    $\mathcal A$ is assumed to have some other point outside $B_\varepsilon(x_0)$.

    By Theorem~\ref{th:PaoSte}, there exists $\rho < \varepsilon$ such that 
    the set $\Sigma_\rho = \overline{\St \setminus B_\rho(\mathcal A)}$ 
    is an embedding of a finite graph.
    Clearly $\Sigma_\rho \cap S_0\neq \emptyset$ because $S_0\subset \overline{\St}$ 
    touches $\partial B_\rho(x_0)$, $\Sigma_\rho \supset \St \cap \partial B_\rho(x_0)$,
    $\overline{\St}  \setminus \St \subset \mathcal A$ (Theorem~\ref{th:PaoSte})
    and $\partial B_\rho\cap \mathcal A = \emptyset$.

    Hence $\Sigma = \Sigma_0\cup S_0$ is connected and contains $x_0$.
    Also $\Sigma \cap \partial B_\rho(x_0) = \Sigma_0\cap \partial B_\rho(x_0)$
    is finite (by Theorem~\ref{th:PaoSte}).
    The set $S=\Sigma\cap \overline{B_\rho}$ is compact and 
    has a finite number of connected components because each connected component has 
    touches $\Sigma \cap \partial B_\rho(x_0)$, which is finite.
    Let $C$ be a connected component of $S$.
    We claim that $C$ is a Steiner set for $C\cap \partial B_\rho$,
    otherwise a Steiner set $C'$ for $C\cap \partial B_\rho$ could be used to construct 
    a competitor $\St' = \St \setminus C \cup C'$ and obtain 
    the contradiction $\H(\St')<\H(\St)$.
    Since $C \cap \partial B_\rho \subset \Sigma \cap \partial B_\rho$ is finite, 
    the set $S$ is a finite embedded tree (Theorem~\ref{th:PaoSte}) and hence 
    $S$ is an embedding of a finite graph.
\end{proof}

\begin{proof}[Proof of Corollary~\ref{cor:isolated}]
For almost every $\rho>0$ Theorem~\ref{th:PaoSte} assures that
$\St \setminus B_\rho(\mathcal A)$ is an embedding of a finite graph.
The set $\mathcal A\setminus B_\rho(\mathcal A')$ is finite because 
it contains only isolated points without any accumulation point.
For each $x_0$ in such finite set we can apply the previous lemma
to find a neighborhood of $x_0$ where $\St$ is an embedding of a finite graph.
If $\varepsilon$ smaller than the minimum of the radii of these neighborhoods, 
we can apply, again, Theorem~\ref{th:PaoSte} to find that also $\St\setminus B_\varepsilon(\mathcal A)$ 
is an embedding of a finite graph and hence deduce 
that $\St \setminus B_\rho(\mathcal A')$ is an embedding of a finite graph.
\end{proof}

\begin{proof}[Proof of Lemma~\ref{lm:Maxwell}]
Let $v_1,\ldots,v_{n+t} \in \mathbb{C}$ be the vertices of $\St$ in the complex plane (they include the terminal points $v_k=p_k$ for $k=1,\dots, n$ and $t \leq n-2$ branching points) and let $E$ be the set of edges 
of $\St$.
Let $u(z) := \frac{z}{\lvert z\rvert} = e^{i\arg z}$ be the unit number representing the direction of the complex number $z$
so that $\lvert z\rvert = z\cdot \overline{u(z)}$.
One has
\begin{align*}
\H (\St) 
&= \sum_{(k,j) \in E} |v_k - v_j| 
= \sum_{(k,j) \in E} (v_k - v_j) \cdot \overline{u(v_k - v_j)}
= \sum_{(k,j)\in E} \left ( v_k \cdot \overline{u(v_k-v_j)} + v_j \cdot \overline{u(v_j-v_k)} \right )\\
& = \sum_{k=1}^{n+t} v_k \cdot \sum_{(k,j)\in E} \overline{u(v_k-v_j)}.
\end{align*}
If $v_k$ is a branching point then the corresponding sum of directions is zero, so we have only a sum over the terminal points.
\end{proof}

\begin{proof}[Proof of Lemma~\ref{lm:Melzak}]
    Clearly, $\angle p_1pp_2 = \pi/3$, $\angle p_1qp_2 = 2\pi/3$ and the line $(p_1,p_2)$ separates $p$ and $q$. Hence $p,p_1,p_2$ and $q$ are cocircular. 
    Then $\angle p_1qp = \angle p_1p_2p = \pi/3$, so $[pq]$ is a continuation of the remaining edge adjacent to $q$ in $\St$.
By the Ptolemy's theorem 
\[
|pq| \cdot |p_1p_2| = |p_1q| \cdot |pp_2| + |p_2q| \cdot |pp_1|,
\]
so $|pq| = |p_1q| + |p_2q|$ and $\H(\St') = \H(\St)$.    
\end{proof}

\begin{proof}[Proof of Theorem~\ref{th:uniq}]
We know that the function $x\mapsto \lvert x \rvert$ is a convex function defined 
on $\mathbb R^d$. 
Hence it turns out that the functional $\ell$ is convex on the vector space $(\mathbb R^d)^V$
of all immersions $\varphi\colon V\to \mathbb R^d$.
So, the real function $F(t) = \ell(t\varphi + (1-t)\psi)$ is convex for $t \in [0,1]$ and if $\varphi$
is a local minimum for $\ell$ it turns out that $t=0$ is a local minimum for $F$. Also, if $v \in V_0$, then $\varphi(v) = \psi(v) = (t\varphi + (1-t)\psi)\ (v)$, so $\ell$ is defined at $t\varphi + (1-t)\psi$. This easily implies that $F$ is constant and hence 
\begin{equation} \label{eq:convexity}
    t\ell(\varphi) + (1-t)\ell(\psi)  = \ell(t\varphi + (1-t) \psi)    
\end{equation}
for every $t \in [0,1]$. Thus $\ell (\varphi) = \ell (\psi)$.

But since the inequality~\eqref{eq:convexity} is true for subgraph of $G$, in particular for every summand in the formula, hence 
equality in\eqref{eq:convexity} holds in each such summand, hence for all $v$ and $w$ which are joined by an
edge of $E$ one has
\[
    t \lvert \varphi(v)-\varphi(w)\rvert + (1-t) \lvert  \psi(v)-\psi(w)\rvert 
    = \lvert t (\varphi(v)-\varphi(w)) + (1-t) (\psi(v)-\psi(w))\rvert
    \qquad \forall t\in[0,1].
\]
But we notice that the equality $t\lvert x\rvert + (1-t) \lvert x\rvert = \lvert tx+(1-t)y\rvert$
holds for all $t \in [0,1]$ only when 
$x$ is a positive multiple of $y$ or when $x=0$ or $y=0$ since the function $\lvert x\rvert$ 
is linear only on the half-lines through the origin.

Letting $x=\varphi(v)-\varphi(w)$ and $y=\psi(v)-\psi(w)$ we obtain the claim of the theorem.
\end{proof}

\begin{proof}[Proof of~\eqref{eq:appendix}]
We claimed that
\[
\frac{\cos\alpha}{\lambda} - \frac{\sin\alpha}{\sqrt{3}\lambda} - \cos \alpha \geq  \frac{\sqrt{3}}{1 - \lambda} \sin \alpha.
\]
This is equivalent to
\[
\cos \alpha \left(\frac{1}{\lambda} - 1 \right) > \sin \alpha \left(\frac{1}{\sqrt{3}\lambda} + \frac{\sqrt{3}}{1-\lambda}\right)
\]
and
\[
\cot \alpha > \frac{1+2\lambda}{\sqrt{3}(1-\lambda)^2}.
\]
One has
\[
\frac{d}{d\lambda} \frac{1+2\lambda}{\sqrt{3}(1-\lambda)^2} = \frac{2(2+\lambda)}{\sqrt{3}(1-\lambda)^3}
\]
which is positive for $0 < \lambda \leq 1/2$.
Cotangent is decreasing, so in view of~\eqref{eq:438466} it is enough to examine
\[
\lambda = \left(\frac{\cos \left(\frac{\pi}{3} + \alpha \right)}{\cos \left(\frac{\pi}{3} - \alpha \right)}\right)^2 = 
\left( \frac{\cot \alpha - \sqrt{3}}{\cot \alpha + \sqrt{3}} \right)^2.
\]
Put $t = \cot\alpha$, then the substitution gives that it is enough to show that
\[
\frac{-\sqrt{3} t^4 + 44t^3 - 2\sqrt{3}t^2 - 12 t - 9\sqrt{3}}{48t^2} > 0
\]
for $t \in [\sqrt{3}, 3\sqrt{3} + 2\sqrt{6}]$, were the ends correspond to $\alpha = \pi/6$ and $\lambda = 1/2$, respectively.
One can check that $-\sqrt{3} t^4 + 44t^3 - 2\sqrt{3}t^2 - 12 t - 9\sqrt{3}$ is increasing on this interval and so the minimal value is $96\sqrt{3}$ at $t = \sqrt{3}$.
    
\end{proof}

\bibliographystyle{plain}
\bibliography{main}

\end{document}

%% file: anglepict/1trees.tex
\begin{tikzpicture}

\definecolor{green}{rgb}{0.0,0.6,0.0}

 \begin{scope}[scale=1.6]
    \def\r{1.5cm}
    \draw[ultra thick, blue]
        (-\r, \r) coordinate(x1)  --++ (-60:{\r/cos(30)}) coordinate (x5);
    \draw[ultra thick, blue]
        (\r,\r) coordinate(x2)  --++ (-120:{\r/cos(30)}) coordinate (x6);
    \draw[ultra thick, blue]
        (\r, -\r) coordinate(x3)  --++ (120:{\r/cos(30)});
    \draw[ultra thick, blue]
        (-\r,-\r) coordinate(x4)  --++ (60:{\r/cos(30)});
    \draw[ultra thick, blue]
        (x5) -- (x6);

\draw[ultra thick, green]
        (-\r, \r) coordinate(x1) --++ (-30:{\r/cos(30)}) coordinate (x7);
    \draw[ultra thick, green]
        (\r,\r) coordinate(x2)  --++ (-150:{\r/cos(30)});
    \draw[ultra thick, green]
        (\r, -\r) coordinate(x3)  --++ (150:{\r/cos(30)}) coordinate (x8);
    \draw[ultra thick, green]
        (-\r,-\r) coordinate(x4)  --++ (30:{\r/cos(30)});
    \draw[ultra thick, green]
        (x7) -- (x8);
    \foreach \x in{1,2,...,4}{
        \draw [fill=red] (x\x) circle (1.5pt);
    }
\end{scope}

   \begin{scope}[shift={(15,0)}]
  
  \def\lambda{0.3}
  \def\ty{0.4} 
  \def\r{7cm}
  \def\rr{\lambda*\r}
  \def\rrr{\ty*\ty*\r}
  \def\rrrr{\ty*\ty*\ty*\r}
  \def\rrrrr{\ty*\ty*\ty*\ty*\r}

    \path (-1.5*\r,0) coordinate  (y0);
    \path (-\r,0) coordinate (y1);

	\path (y1) ++(60:0.75*\ty*\r) coordinate  (y2);
	\path (y1) ++(-60:0.75*\ty*\r) coordinate  (y3);
	\path (y2) ++(120:\ty*\rr) coordinate  (y4);
	\path (y2) ++(0:\ty*\rr) coordinate (y5);
	\path (y3) ++(0:\ty*\rr) coordinate (y6);
	\path (y3) ++(-120:\ty*\rr) coordinate (y7);
	\path (y4) ++(180:\ty*\rrr) coordinate (y8);
	\path (y4) ++(60:\ty*\rrr) coordinate (y9);
	\path (y5) ++(60:\ty*\rrr) coordinate (y10);
	\path (y5) ++(-60:\ty*\rrr) coordinate (y11);
	\path (y6) ++(60:\ty*\rrr) coordinate (y12);
	\path (y6) ++(-60:\ty*\rrr) coordinate (y13);
	\path (y7) ++(-60:\ty*\rrr) coordinate (y14);
	\path (y7) ++(180:\ty*\rrr) coordinate (y15);

    \path (y8) ++(-120:\ty*\rrrr) coordinate (y16);
	\path (y8) ++(120:\ty*\rrrr) coordinate (y17);
	\path (y9) ++(120:\ty*\rrrr) coordinate (y18);
	\path (y9) ++(0:\ty*\rrrr) coordinate (y19);
	\path (y10) ++(120:\ty*\rrrr) coordinate (y20);
	\path (y10) ++(0:\ty*\rrrr) coordinate (y21);
	\path (y11) ++(0:\ty*\rrrr) coordinate (y22);
	\path (y11) ++(-120:\ty*\rrrr) coordinate (y23);
        \path (y12) ++(120:\ty*\rrrr) coordinate (y24);
	\path (y12) ++(0:\ty*\rrrr) coordinate (y25);
	\path (y13) ++(0:\ty*\rrrr) coordinate (y26);
	\path (y13) ++(-120:\ty*\rrrr) coordinate (y27);
	\path (y14) ++(0:\ty*\rrrr) coordinate (y28);
	\path (y14) ++(-120:\ty*\rrrr) coordinate (y29);
	\path (y15) ++(-120:\ty*\rrrr) coordinate (y30);
	\path (y15) ++(120:\ty*\rrrr) coordinate (y31);

  \def\fork{\path[draw, color=blue, ultra thick]}

    \fork (y2) -- (y1) -- (y3);
    \fork (y4) -- (y2) -- (y5);
    \fork (y6) -- (y3) -- (y7);

   \path[draw, color=blue, ultra thick] (y0) -- (y1);

  \foreach \n [evaluate=\n as \m using {int(2*\n)}] in {4,5,6,7} { 
    \path[draw, color=blue, ultra thick, dotted] (y\n) -- (y\m);
  }

 \foreach \n [evaluate=\n as \m using {int(2*\n+1)}] in {4,5,6,7} { 
    \path[draw, color=blue, ultra thick, dotted] (y\n) -- (y\m);
  }

  \draw [fill=red] (y0) circle (2pt);

  \foreach \n in {1,2,...,15} {
    \draw [fill=blue] (y\n) circle (2pt);
  }
  
  \foreach \n in {16,17,...,31} {
    \draw [fill=red] (y\n) circle (2pt);
  }
 \end{scope}

\end{tikzpicture}

%% file: anglepict/2defA.tex
\begin{tikzpicture}[scale=10,rotate=-5]

\coordinate [label={[red,above]:$A_\infty$}] (ab) at (0,0);

\draw [ultra thick] (ab) --++ (10:0.031) coordinate (a6) --++ (10:0.031) coordinate (a5) --++ (10:0.063) coordinate (a4) --++ (10:0.125) coordinate [label={[red,above left]:$A_{3}$}] (a3) --++ (10:0.25) coordinate [label={[red,above left]:$A_{2}$}] (a2)
--++ (10:0.5) coordinate [label={[red,above left]:$A_{1}$}] (a1) --++ (10:0.5) coordinate [label={[red,above left]:$A_{0}$}] (a0);

\draw [ultra thick] (ab) --++ (0:0.031) coordinate (b6) --++ (0:0.031) coordinate (b5) --++ (0:0.063) coordinate (b4) --++ (0:0.125) coordinate [label={[red,below left]:$B_{3}$}] (b3) --++ (0:0.25) coordinate [label={[red,below left]:$B_{2}$}] (b2)
--++ (0:0.5) coordinate [label={[red,below left]:$B_{1}$}] (b1) --++ (0:0.5) coordinate [label={[red,below left]:$B_{0}$}] (b0);

\draw[ultra thick, red] (a0) -- (b0);

\foreach \n in {0,1,...,6} {
    \draw [fill=red] (a\n) circle (0.2pt);
    \draw [fill=red] (b\n) circle (0.2pt);
  }
\draw [fill=red] (ab) circle (0.2pt);

\end{tikzpicture}

%% file: anglepict/3.3mainT.tex
\begin{tikzpicture}[scale = 0.7,rotate=-5]

\coordinate [label={[red,above]:$A_\infty$}] (ab) at (0,0);

\draw [ultra thick] (ab) --++ (10:2.5) coordinate  (a3) --++ (10:2.5) coordinate [label={[red,above left]:$A_{3}$}] (a2)
--++ (10:5) coordinate [label={[red,above left]:$A_{2}$}] (a1) --++ (10:10) coordinate  [label={[red,above left]:$A_{1}$}] (a0) --++ (10:1);

\draw [ultra thick] (ab) --++ (0:2.5) coordinate  (b3) --++ (0:2.5) coordinate [label={[red,below left]:$B_{3}$}] (b2)
--++ (0:5) coordinate [label={[red,below left]:$B_{2}$}] (b1) --++ (0:10) coordinate  [label={[red,below left]:$B_{1}$}] (b0) --++ (0:1);

\draw[ultra thick, blue] (a0) --++ (-113.7:2.09) coordinate (y0) -- (b0);
\draw[ultra thick, blue] (y0) --++ (-173.7:8.65) coordinate (y1) -- (b1);
\draw[ultra thick, blue] (y1) --++ (126.3:0.967) coordinate (y2) -- (a1);
\draw[ultra thick, blue] (y2) -- (a2);

\foreach \n in {0,1,2} {
    \draw [fill=red] (a\n) circle (2pt);
    \draw [fill=red] (b\n) circle (2pt);
  }
\draw [fill=red] (ab) circle (2pt);  
\draw [fill=blue] (y0) circle (2pt);  
\draw [fill=blue] (y2) circle (2pt);
\draw [fill=blue] (y1) circle (2pt);

\begin{scope}[scale=0.25]
    \draw [ultra thick] (ab) --++ (10:2.5) coordinate  (a3) --++ (10:2.5) coordinate [label={[red,above left]:$A_{5}$}] (a2)
--++ (10:5) coordinate [label={[red,above left]:$A_{4}$}] (a1) --++ (10:10) coordinate  (a0) --++ (10:1);

\draw [ultra thick] (ab) --++ (0:2.5) coordinate  (b3) --++ (0:2.5) coordinate [label={[red,below left]:$B_{5}$}] (b2)
--++ (0:5) coordinate [label={[red,below left]:$B_{4}$}] (b1) --++ (0:10) coordinate  (b0) --++ (0:1);

\draw[ultra thick, blue] (a0) --++ (-113.7:2.09) coordinate (y0) -- (b0);
\draw[ultra thick, blue] (y0) --++ (-173.7:8.65) coordinate (y1) -- (b1);
\draw[ultra thick, blue] (y1) --++ (126.3:0.967) coordinate (y2) -- (a1);
\draw[ultra thick, blue] (y2) -- (a2);

\foreach \n in {0,1,2} {
    \draw [fill=red] (a\n) circle (8pt);
    \draw [fill=red] (b\n) circle (8pt);
  }
\draw [fill=red] (ab) circle (2pt);  
\draw [fill=blue] (y0) circle (2pt);  
\draw [fill=blue] (y2) circle (2pt);
\draw [fill=blue] (y1) circle (2pt);
    
\end{scope}

\begin{scope}[scale=0.0625]
    \draw [] (ab) --++ (10:2.5) coordinate  (a3) --++ (10:2.5) coordinate (a2)
--++ (10:5) coordinate (a1) --++ (10:10) coordinate  (a0) --++ (10:1);

\draw [] (ab) --++ (0:2.5) coordinate  (b3) --++ (0:2.5) coordinate  (b2)
--++ (0:5) coordinate (b1) --++ (0:10) coordinate  (b0) --++ (0:1);

\draw[, blue] (a0) --++ (-113.7:2.09) coordinate (y0) -- (b0);
\draw[, blue] (y0) --++ (-173.7:8.65) coordinate (y1) -- (b1);
\draw[, blue] (y1) --++ (126.3:0.967) coordinate (y2) -- (a1);
\draw[, blue] (y2) -- (a2);

\foreach \n in {0,1,2} {
    \draw [fill=red] (a\n) circle (32pt);
    \draw [fill=red] (b\n) circle (32pt);
  }
\draw [fill=red] (ab) circle (2pt);  
\draw [fill=blue] (y0) circle (2pt);  
\draw [fill=blue] (y2) circle (2pt);
\draw [fill=blue] (y1) circle (2pt);
    
\end{scope}

\end{tikzpicture}

%% file: anglepict/3theorem.tex
\begin{tikzpicture}[scale=10,rotate=-5]

\coordinate [label={[red,above]:$A_\infty$}] (ab) at (0,0);

\draw [ultra thick] (ab) --++ (10:0.031) coordinate (a6) --++ (10:0.031) coordinate (a5) --++ (10:0.063) coordinate (a4) --++ (10:0.125) coordinate [label={[red,above left]:$A_{3}$}] (a3) --++ (10:0.25) coordinate [label={[red,above left]:$A_{2}$}] (a2)
--++ (10:0.5) coordinate [label={[red,above left]:$A_{1}$}] (a1) --++ (10:0.5) coordinate [label={[red,above left]:$A_{0}$}] (a0);

\draw [ultra thick] (ab) --++ (0:0.031) coordinate (b6) --++ (0:0.031) coordinate (b5) --++ (0:0.063) coordinate (b4) --++ (0:0.125) coordinate [label={[red,below left]:$B_{3}$}] (b3) --++ (0:0.25) coordinate [label={[red,below left]:$B_{2}$}] (b2)
--++ (0:0.5) coordinate [label={[red,below left]:$B_{1}$}] (b1) --++ (0:0.5) coordinate [label={[red,below left]:$B_{0}$}] (b0);

\draw[ultra thick, red] (a0) -- (b0);

\coordinate [label={[red, above left]:$x$}] (y0) at ($(a0)!+0.7222!(b0)$);

\draw[ultra thick, blue] (y0) --++ (185:0.4805) coordinate (y1) --++ (245:0.03);
\draw[ultra thick, blue] (y1) --++ (125:0.102) coordinate (y2) --++ (65:0.06);
\draw[ultra thick, blue] (y2) --++ (185:0.457) coordinate (y3) --++ (125:0.01);
\draw[ultra thick, blue] (y3) --++ (245:0.05) coordinate (y4) --++ (305:0.03);
\draw[ultra thick, blue] (y4) --++ (185:0.225) coordinate (y5) --++ (245:0.01);
\draw[ultra thick, blue] (y5) --++ (125:0.026) coordinate (y6) --++ (65:0.015);
\draw[ultra thick, blue, dotted] (y6) --++ (185:0.05);


\foreach \n in {0,1,...,6} {
    \draw [fill=red] (a\n) circle (0.2pt);
    \draw [fill=red] (b\n) circle (0.2pt);
  }
\draw [fill=red] (ab) circle (0.2pt);

\end{tikzpicture}

%% file: anglepict/3.5klemma.tex
\begin{tikzpicture}[scale=10,rotate=-5]

\definecolor{green}{rgb}{0.0,0.6,0.0}

\coordinate [label={[red,above]:$A_\infty$}] (ab) at (0,0);

\draw [ultra thick] (ab) --++ (10:0.5) coordinate [label={[red,above left]:$A_{k+1}$}] (a2) --++ (10:0.5) coordinate [label={[red,above left]:$A_{k}$}] (a1) --++ (10:0.25);

\draw [ultra thick] (ab) --++ (0:0.5) coordinate [label={[red,above left]:$B_{k+1}$}] (b2) --++ (0:0.5) coordinate [label={[red,above left]:$B_{k}$}] (b1) --++ (0:0.25);

\coordinate [label={[above]:$W_{1}$}] (c) at (0.707,0);

\draw (a2) -- (c) -- (a1);

\draw pic[mydeco,draw=green,angle radius=0.6cm] {angle={a2--c--ab}};
\draw pic[mydeco,draw=green,angle radius=0.7cm] {angle={a2--c--ab}};

\draw pic[mydeco,draw=green,angle radius=0.6cm] {angle={ab--a1--c}};
\draw pic[mydeco,draw=green,angle radius=0.7cm] {angle={ab--a1--c}};


\foreach \n in {1,2} {
    \draw [fill=red] (a\n) circle (0.2pt);
    \draw [fill=red] (b\n) circle (0.2pt);
  }
\draw [fill=red] (ab) circle (0.2pt);
\fill  (c) circle (0.2pt);

\end{tikzpicture}

%% file: anglepict/4paralelograms.tex
\begin{tikzpicture}[scale=10,rotate=-5]

\coordinate [label={[red,above]:$A_\infty$}] (ab) at (0,0);

\draw [ultra thick] (ab) --++ (10:0.25) coordinate  (a3) --++ (10:0.25) coordinate [label={[red,above left]:$A_{k+1}$}] (a2)
--++ (10:0.5) coordinate [label={[red,above left]:$A_{k}$}] (a1) --++ (10:0.5) coordinate (a0);

\draw [ultra thick] (ab) --++ (0:0.25) coordinate  (b3) --++ (0:0.25) coordinate [label={[red,below left]:$B_{k+1}$}] (b2)
--++ (0:0.5) coordinate [label={[red,below left]:$B_{k}$}] (b1) --++ (0:0.5) coordinate  (b0);


\coordinate (y0) at ($(a0)!+0.7222!(b0)$); 

\draw[white] (y0) --++ (185:0.4805) coordinate (y1) --++ (245:0.03);
\draw (b1) --++ (65:0.09) coordinate (v) -- (a1) --++ (245:0.098) coordinate (u) -- (b1);

\draw[ultra thick, blue] (y1) --++ (125:0.102) coordinate (y2) --++ (65:0.06);
\draw[ultra thick, blue] (y2) --++ (185:0.157);
\draw[ultra thick, blue] (y1) -- (b1);
\draw[ultra thick, blue] (y1) -- ($(y1)!+0.222!(y0)$);



\foreach \n in {1,2} {
    \draw [fill=red] (a\n) circle (0.2pt);
    \draw [fill=red] (b\n) circle (0.2pt);
  }
\draw [fill=red] (ab) circle (0.2pt);  
\draw [fill=blue] (y2) circle (0.2pt) node[right]{$S_k$};
\draw [fill=blue] (y1) circle (0.2pt) node[left]{$T_k$};
\draw [fill=] (v) circle (0.2pt) node[right]{$U_k$};
\draw [fill=] (u) circle (0.2pt) node[left]{$V_k$};

\end{tikzpicture}

%% file: anglepict/6observations.tex
\begin{tikzpicture}[scale = 0.7, rotate = -5]

\coordinate [label={[red,above]:$A_\infty$}] (ab) at (0,0);

\draw [ultra thick] (ab) --++ (10:2.5) coordinate  (a3) --++ (10:2.5) coordinate [label={[red,above left]:$A_{3}$}] (a2)
--++ (10:5) coordinate [label={[red,above left]:$A_{2}$}] (a1) --++ (10:10) coordinate  [label={[red,above left]:$A_{1}$}] (a0) --++ (10:1);

\draw [ultra thick] (ab) --++ (0:2.5) coordinate  (b3) --++ (0:2.5) coordinate [label={[red,below left]:$B_{3}$}] (b2)
--++ (0:5) coordinate [label={[red,below left]:$B_{2}$}] (b1) --++ (0:10) coordinate  [label={[red,below left]:$B_{1}$}] (b0) --++ (0:1);

\draw[thick] (a1) --++ (-52:0.89) --++ (-112:1.1);
\draw[thick] (b1) --++ (128:0.89) --++ (68:1.1);

\draw[thick,dotted] (a0) --++ (-25:3.486) coordinate  [label={[right]:$C$}] (c) -- (b0) -- (a0);
\draw[dotted] (c) -- (ab);

\draw[ultra thick, blue] (a0) --++ (-112:2.16) coordinate (y0) -- (b0);
\draw[ultra thick] (y0) -- (c);
\draw[ultra thick, blue] (y0) --++ (-172:8.88) coordinate (y1) -- (b1);
\draw[ultra thick, blue] (y1) --++ (128:0.89) coordinate (y2) -- (a1);
\draw[ultra thick, blue] (y2) --++ (188:2.89);

\foreach \n in {0,1,2} {
    \draw [fill=red] (a\n) circle (2pt);
    \draw [fill=red] (b\n) circle (2pt);
  }
\draw [fill=red] (ab) circle (2pt);  
\draw [fill=blue] (y0) circle (2pt) node[below]{$S_1$};  
\fill (c) circle (2pt);  
\draw [fill=blue] (y2) circle (2pt) node[above left]{$S_2$};
\draw [fill=blue] (y1) circle (2pt) node[above right]{$T_2$};

\end{tikzpicture}

%% file: anglepict/7dynamics.tex
\begin{tikzpicture}[scale=20]

\definecolor{green}{rgb}{0.0,0.6,0.0}





 \clip (0.85,-0.01) rectangle (1.15,0.18);

\coordinate (a1) at (0.985,0.173);
\coordinate (b1) at (1,0);
\coordinate (a0) at (1.47,0.260);
\coordinate (b0) at (1.5,0);

\coordinate (y0) at ($(a0)!+0.7222!(b0)$);

\coordinate (c) at ($(a1)!+0.5!(b1)$); 

\path (y0) --++ (185:0.478) coordinate (y1);
\draw (b1) --++ (65:0.098) coordinate (u) -- (a1) --++ (245:0.098) coordinate (v) -- (b1);

\draw[ultra thick, blue] (y1) --++ (125:0.099) coordinate (y2) --++ (65:0.063);
\draw[ultra thick, blue] (y2) --++ (185:0.157);
\draw[ultra thick, blue] (y1) -- (b1);
\draw[ultra thick, blue] (y1) -- ($(y1)!+0.222!(y0)$);

\coordinate (z0) at ($(b0)!+0.65!(a0)$); 
\path (z0) --++ (185:0.4673) coordinate (z1);
\draw[ultra thick, green] (z1) -- (a1);
\draw[ultra thick, green] (z1) --++ (-115:0.0965) coordinate (z2) -- (b1);
\draw[ultra thick, green] (z2) --++ (185:0.157);
\draw[ultra thick, green] (z1) -- ($(z1)!+0.222!(z0)$);




    \draw [fill=red] (a1) circle (0.15pt) node[right]{$A_k$};
    \draw [fill=red] (b1) circle (0.15pt) node[right]{$B_k$};


\draw [fill=] (u) circle (0.1pt) node[right]{$U_k$};
\draw [fill=] (v) circle (0.1pt) node[left]{$V_k$};

\draw [fill=] (c) circle (0.1pt) node[right]{$C_k$};


\end{tikzpicture}